%% file: 2024-05-29-arxiv.tex
\documentclass[10pt]{article}

\usepackage{authblk}
\usepackage{latexsym,hyperref}
\usepackage{verbatim}
\usepackage[noadjust]{cite}
\usepackage[english]{babel}
\usepackage[latin1]{inputenc}
\usepackage{enumitem}
\usepackage{amsmath,amsthm}
\usepackage{mathtools}

\usepackage{multirow}

\usepackage[usenames, dvipsnames]{color}

\usepackage{tikz}
\usetikzlibrary{tikzmark}
\usepackage{amsfonts}
\usepackage{amssymb}
\usepackage{amsbsy}

\usepackage{bm}

    \let\leq\leqslant
    \let\geq\geqslant
    \let\emptyset\varnothing

\input{ka-newcommands}

\input{ka-newcommands-extra}

\newcounter{todocounter}
\usepackage{todonotes}

\makeatletter

\newcommand*\@bigplus[1]{\vcenter{\hbox{#1$\m@th +$}}}
\newcommand*\bigplus{%
    \DOTSB 
    \mathop{%
        \mathchoice
            {\@bigplus \huge}%
            {\@bigplus \LARGE}%
            {\@bigplus {}}%
            {\@bigplus \footnotesize}%
    }%
    \slimits@ 
}

\makeatother

\newtheorem{theorem}{Theorem}
\newtheorem{lemma}[theorem]{Lemma}

\newtheorem{assumption}[theorem]{Assumption}
\newtheorem{proposition}[theorem]{Proposition}
\newtheorem{corollary}[theorem]{Corollary}
\newtheorem{remark}[theorem]{Remark}
\newcommand{\brem}{\begin{remark}}
\newcommand{\erem}{\end{remark}}

\theoremstyle{definition}
\newtheorem{definition}[theorem]{Definition}
\newtheorem{example}[theorem]{Example}

\newtheorem{problem}[theorem]{Problem}
\newtheorem{conjecture}[theorem]{Conjecture}

\newcommand{\benabc}{\begin{enumerate}[label={\em(\alph*)}, ref=(\alph*)]}
\newcommand{\beni}{\begin{enumerate}[label={\em(\roman*)}, ref=(\roman*)]}
\newcommand{\benabcrm}{\begin{enumerate}[label=(\alph*), ref=(\alph*)]}

\usepackage[a4paper, total={7in, 10in}]{geometry}

\begin{document}
%
\title{Beyond the fundamental lemma:\\
 from finite time series to linear system}
%
%
%

\author[1]{M.~Kanat~Camlibel}
\author[2]{Paolo~Rapisarda}
\affil[1]{Bernoulli Institute, University of Groningen}

\affil[2]{School of Electronics and Computer Science, University of Southampton}

\date{}
\maketitle

\begin{abstract}
We state necessary and sufficient conditions to uniquely identify (modulo state isomorphism) a linear time-invariant minimal input-state-output  system from finite input-output data and upper- and lower bounds on  lag and state space dimension. 
\end{abstract}

%

\section{Introduction}
\noindent{\em {Background.}}
J.C.~Willems'  trilogy \cite{Willems86a,Willems86b,Willems87} is one of broader, deeper, and more influential studies about mathematical modelling of dynamical systems from  time series. The second part \cite{Willems86b} concerns the problem of obtaining a mathematical model for a linear system from a given (infinite) trajectory. It   
significantly influenced  subspace identification methods \cite{Moonen1989}, that compute a state sequence from finite-length data by adapting Willems' state construction  from infinite-  to finite-length data.  Two  assumptions are crucial  in \cite{Moonen1989}: the 
state space dimension of the  system is known; and a rank condition holds for a Hankel matrix constructed from the   data. Although not formally proven at the time, it was  believed that such rank condition is satisfied if the input data are sufficiently persistently  exciting. This conjecture was formally proven in Willems {et. al.}'s \emph{fundamental lemma} (\cite[Thm. 1]{Willems05}) which allows the application of subspace identification even when only an upper bound on the  state  dimension is known. 

The fundamental lemma parameterizes  {\em all\/} trajectories of a system from a {\em single} sufficiently informative one. This parameterization was applied in linear quadratic control \cite{Markovsky2007}, simulation \cite{Markovsky2008}, model reduction of dissipative systems \cite{Rapisarda2011}, and predictive control \cite{Yang2013}. 
More recently, this approach has gained significant momentum, initiated by papers such as
\cite{Maupong2017,Coulson19,DePerzik20,Berberich21} and followed by many  contributions addressing a variety of  data-driven  analysis and control  problems.

The recent surge in popularity of this parameterization has also revived the interest in the fundamental lemma itself. Its original proof was presented in the language of behavioral systems; an alternative proof for state space systems was provided in \cite{Waarde20b}. The original proof and that in \cite{Waarde20b} are by contradiction; a direct proof was presented in \cite{Markovsky23a} for single-input systems. Generalizations to uncontrollable systems are in \cite{Yu21,Mishra21,Markovsky23a} and extensions to continuous-time systems in \cite{Lopez22,Rapisarda23a,Rapisarda23b,Lopez24}. Quantitative/robust variations are explored in \cite{Berberich23,Coulson23}, frequency domain formulations in \cite{Ferizbegovic21,Meijer23}, and online experiment design in \cite{Waarde21}. Furthermore, the fundamental lemma has been generalized beyond linear systems to include various other model classes: descriptor  systems \cite{Schmitz22}, flat nonlinear systems \cite{Alsalti21}, linear parameter-varying systems \cite{Verhoek21}, and stochastic ones \cite{Faulwasser22}.\medskip 

\noindent{\em Contributions.} The fundamental lemma gives a \emph{sufficient} condition for identifiability (in the sense of \cite{Willems05,Markovsky05}); in this paper we investigate {\em necessary and sufficient\/} conditions on {\em finite input-output\/} data for identifiability of an unknown minimal input-state-output (ISO) system whose lag and state space dimension lie between given lower and upper bounds. Our main contributions are:
\begin{itemize}
\item Using the concept of \emph{data informativity} (see \cite{Waarde20a,Waarde23-CSM}), we develop a framework for ISO system identification from finite input-output data, incorporating a priori knowledge or assumptions (Section~\ref{sec:probform}). While currently limited to exact (deterministic) identification problems, such framework is of broader potential interest, for example in approximate modelling and for noisy data. 

\item We characterize (Theorem~\ref{t:nmin}) the shortest lag $\lm$ and the minimum number of states $\nm$, that an ISO system capable of generating the data must have. These integers can be computed \emph{directly from the data}.

\item We compute an ISO system whose lag and state  dimension are exactly $\lm$ and $\nm$ using   a novel construction of a state trajectory from  data (Theorem~\ref{t:state construct}) that does not rely on assumptions on the length of the data set. 

\item We establish an inequality (see \eqref{e: l bound with n lm nm}) relating the lag and state dimension of {\em any\/} data-generating system to $\lm$ and $\nm$. With such inequality we compute a sharper upper bound on the lag than  the one given a priori.

\item Assuming a priori knowledge about minimality and lower and upper bounds on  lag and state dimension, we state \emph{necessary and sufficient} conditions on finite input-output data to guarantee  unique identification modulo state isomorphism (Theorem~\ref{t:main}). Such conditions are formulated in terms of the rank of a Hankel matrix constructed from the input-output data, whose depth  is determined by the  input-output data themselves.  
\end{itemize}\medskip

\noindent{\em Relation with previous results.}
As in deterministic  {subspace identification} methods (see  \cite{Moonen1989,vanOverschee1996,Markovsky05c}), our approach is also based on  directly computing  a state trajectory  from  input-output data. 
Our procedure (see Section \ref{sec:construct state}) is also related to the intersection of ``past" and ``future" crucial in subspace identification  (see \cite[Sect. 2.3.1]{vanOverschee1996}). The issues arising when working with \emph{finite-length} data, however, seem to have been only touched upon in the subspace literature (see \cite[p. 34]{vanOverschee1996} and  before the statement \cite[Thm. 2]{Moonen1989}). It is   assumed that the data length is ``sufficiently large" to guarantee that the data Hankel matrix contains enough information on the dynamics of an explaining model, but a \emph{complete} characterization of such property such as given in the present paper is absent. 

Our approach is conceptually and methodologically closest to the behavioral one (see \cite{Willems86b,Willems05,Markovsky23}). Instrumental to our  results is the definition of a number of integer invariants computed directly from the data and associated with explaining models. Such integers are the finite data counterparts of those introduced in \cite[Sect. 7]{Willems86a} for {\em infinite\/} time series. Moreover, the fundamental lemma  is a \emph{special case} of our  results (see Proposition~\ref{prop:wetal}, Theorem~\ref{t:main}, and Section~\ref{sec:prop to main thm}), and we  show that one can identify the unknown system under {\em weaker\/} conditions. Finally, we essentially improve the fundamental lemma as follows. Identification  based on the fundamental lemma is  {\em offline\/}: a sufficiently rich input signal is applied, the corresponding output response is measured, and then identification or trajectory-parametrization is performed. Such {offline\/} method does not exploit in \emph{real-time} information from the output samples. In \cite{Camlibel24} we use our approach  to devise a {\em shortest online experiment} for linear system identification. Input values are chosen depending on the past input-output data and are applied \emph{step-by-step}. Such  {online\/} method   requires  less data points than the {offline\/} one, and it  could pave the way to design plug-and-play data-driven controllers. 

The authors of \cite{Markovsky23}  study identifiability from multiple finite-length trajectories without assumptions on the input-output structure.  The first  similarity with our approach lies in the adoption of the ``model class'' concept; in \cite{Markovsky23} such class is associated with the \emph{complexity} of a model (see \cite[Section III.B]{Markovsky23}). A  second similarity lies in the sequential construction of left-annihilators  of finite Hankel data matrices, although in \cite{Markovsky23} an autoregressive model is computed and we compute an ISO one instead.  A  fundamental difference is that to compute such annihilators only   data Hankel matrices with at least as many columns as rows are used in \cite{Markovsky23} (see the definition of $L_{\tiny \rm max}$ on p. 3 therein); our informativity point of view instead allows us to exploit Hankel matrices of \emph{full depth} in  constructing  a state sequence compatible with the data. Moreover, the identifiability characterization in \cite[Theorem 17]{Markovsky23} is based on a priori knowledge of  lag and state dimension (see also Note 18 \emph{ibid.}); ours depends on integers computed \emph{only} and \emph{directly} from the data. 
\medskip

\noindent{\em Structure of the paper.} In Section~\ref{sec:notation}, we introduce  notation,  preliminary concepts and definitions. In Section~\ref{sec:probform} we formalize the problem.  Section~\ref{s: summary main results} contains a summary of our most relevant results and an outline of the logical links among them. In Section~\ref{sec:running example} we illustrate such results on an example. 

Our main result is Theorem~\ref{t:main} to whose proof, based on some intermediate results of independent interest, we devote the rest of the manuscript. In Section~\ref{sec:lag struct} we introduce and study lag structures of explaning systems. Section~\ref{sec:construct state} contains a new iterative state construction from input-output data, from which an explaining system is straightforwardly derived. State-input data Hankel matrices and their properties are studied in Section~\ref{sec:SI Hankel}. In Section~\ref{sec: from one to another} we present two  ways of constructing an explaining system from a given one. We discuss future work in Section~\ref{sec:conc}.

\section{Notation and preliminaries}\label{sec:notation}

\subsection{Vectors and matrices}
The space of $n$-dimensional real vectors is denoted by $\R^n$; the space of $n\times m$ matrices with real entries by $\R^{n\times m}$. 

We denote the $n\times n$ identity matrix by $I_n$ and the $m\times n$ zero matrix by $0_{m,n}$ whereas $0_n$ denotes the $n$-vector of zeros. For partitioned matrices containing zero and/or identity submatrices, we do not explicitly indicate the sizes of blocks that can be deduced from the matrix structure.  
  
Given matrices $M_i$, $i=1,2,\ldots,k$ with the same number of columns, we denote $\begin{bmatrix}M_1^T&M_2^T&\cdots&M_k^T  \end{bmatrix}^T$ by $\col\left(M_1,M_2,\ldots,M_k\right)$. Given $M\in\R^{m\times n}$, we denote its {\em kernel\/} by $\ker M:=\set{x\in\R^n}{Mx=0_m}$, its \emph{row space} by $\rs M:=\set{vM}{v\in\R^{1\times m}}$ and its \emph{left kernel} by 
$\lk M:=\set{v\in\R^{1\times m}}{vM=0_{1,n}}.$

We denote the zero subspace of $\R^{1\times n}$ by $\bz_n$.

\subsection{Void matrices} A {\em void matrix\/} is a matrix with zero rows and/or zero columns. We denote by  $0_{n,0}$ and $0_{0,m}$ respectively the $n\times 0$ and $0\times m$ void matrices. If $M$ and $N$ are, respectively $p\times q$ and $q\times r$ matrices, $MN$ is a $p\times r$ void matrix if $p=0$ or $r=0$ and $MN=0_{p,r}$ if $p,r\geq 1$ and $q=0$. The  rank of a void matrix is defined to be zero.

\subsection{Integer intervals and Hankel matrices} 
The set of integers is denoted by $ \Z $ and the set of nonnegative integers by $\N$. 

Given $i,j\in\Z$ with $i\leq j$, we write $[i,j]$ to denote the ordered set of all integers between $i$ and $j$ both included. By convention, $[i,j]=\emptyset$ if $i>j$. 

Let $i,j\in\N$, $i\leq j$; let $f_k\in\R^n$ with $k\in[i,j]$. We define
$$
f_{[i,j]}:=\bbm f_i&f_{i+1}&\cdots& f_j\ebm\in\R^{n\times (j-i+1)}.
$$
For $k\in[0,j-i]$, the \emph{Hankel matrix with depth $k+1$ associated with $f_{[i,j]}$} is defined by:
\[
H_k(f_{[i,j]}):=\bbm f_{i}&\cdots&f_{j-k}\\
\vdots&&\vdots\\
f_{i+k}&\cdots&f_{j}\ebm.
\]

\subsection{Input-state-output systems} We work with linear discrete-time ISO systems
\bse\label{eq: generic linear system}
\begin{align}
\bm x(t+1)&=A \bm x(t)+B\bm u(t)\\
\bm y(t)&=C \bm x(t)+D \bm u(t)
\end{align}
\ese
where $n\geq 0$, $m,p\geq 1$, $A\in\Rnn$, $B\in\Rnm$, $C\in\Rpn$, and $D\in\Rpm$. If $n=0$, we call the system {\em memoryless}.

For $k\geq -1$, we define the \emph{$k$-th observability matrix} by
\beq\label{e:def omega-k}
\Omega_k:=\begin{cases}
0_{0,n}&\mbox{ if } k=-1\\
\bbm \Omega_{k-1}\\CA^{k}\ebm&\mbox{ if } k\geq 0
\end{cases} 
\eeq
We denote the smallest integer $k\geq 0$ such that $\rank \Omega_k=\rank \Omega_{k-1}$ by $\ell(C,A)$. Note that $0\leq \ell(C,A)\leq n$; if $n=0$, then $\ell(C,A)=0$. If $(C,A)$ is observable, then $\ell(C,A)$ is the observability index.  We call $\ell(C,A)$ the \emph{lag} of the system; on  this terminology, see statement (vii) of \cite[Thm. 6]{Willems86a}. 

\subsection{Systems with $m$ inputs and $p$ outputs}
We associate with \eqref{eq: generic linear system} the matrix $\sa$. Given $m\geq 1$ and $p\geq 1$, we denote the set of all systems with lag $\ell$ and $n$ state variables by 
\[
\calS(\ell,n):=\set{\sa\in\R^{(n+p)\times(n+m)}}{\ell(C,A)=\ell}\; .
\]
We also define
\begin{align*}
\calS(n)&:=\set{\sa\in\calS(\ell,n)}{\ell\in[0,n]}\\
\calS&:=\set{\sa\in\calS(n)}{n\geq 0}\\
\calO&:=\set{\sa\in\calS}{(C,A)\text{ is observable}}\\
\calM&:=\set{\sa\in\calO}{(A,B)\text{ is controllable}}.
\end{align*}

\subsection{Isomorphic systems} Two systems $\sai\in\calS(n)$, $i=1,2$  are {\em isomorphic\/} if $D_1=D_2$ and there exists a nonsingular matrix $S\in\Rnn$ such that $A_1=S\inv A_2 S$, $B_1=S\inv B_2$, $C_1=C_2 S$. We say that $\calS'\subseteq\calS(n)$ has the {\em isomorphism property\/} if all systems belonging to $\calS'$ are isomorphic to each other. By convention, the empty set has the isomorphism property. 

\section{Problem formulation}\label{sec:probform}
Consider a linear discrete-time input-state-output system
\bse\label{e:true sys}
\begin{align}
\bm x(t+1)&=\ta \bm x(t)+\tb \bm u(t)\\
\bm y(t)&=\tc \bm x(t)+\td \bm u(t)
\end{align}
\ese
%
where $\nt\geq 0$, $\ta \in\R^{\nt\times \nt}$, $\tb\in\R^{\nt\times m}$, 
$\tc\in\R^{p\times \nt}$, and $\td\in\Rpm$. We refer to \eqref{e:true sys} as the {\em true system}. We denote its lag by $\lt:=\ell(\tc,\ta)$. Throughout the paper, we assume that the true system is minimal, i.e. both observable and controllable.

Let $T\geq 1$ and $(\duf,\dyf)$ be input-output data generated by \eqref{e:true sys}: there exists $\dxf\in\R^{\nt\times(T+1)}$ such that
\beq\label{e:data generated}
\bbm\dxfp\\\dyf\ebm
=\bbm \ta& \tb\\\tc& \td\ebm
\bbm\dxfm\\\duf\ebm\; .
\eeq
A fundamental problem in system identification is under what conditions and how the  system \eqref{e:true sys} can be uniquely determined   from $(\duf,\dyf)$ (up to state  isomorphisms). In this paper we examine such questions assuming {\em prior knowledge\/} about the unknown system \eqref{e:true sys}. To formalize such problem we need to introduce some terminology and notation.

\subsection{Explaining systems}
A system $\sa\in\calS(n)$ {\em explains\/} $(\duf,\dyf)$ if there exists $\dxf\in\R^{n\times(T+1)}$ such that
\beq\label{e:explains}
\bbm\dxfp\\\dyf\ebm
=\bbm A& B\\C& D\ebm
\bbm\dxfm\\\duf\ebm.
\eeq
The set of all systems that explain the data $(\duf,\dyf)$ is denoted by $\calE$ and is called the set of {\em explaining systems}. The subsets of $\calE$ consisting of  systems with a given lag and state space dimension are denoted by
\[
\syl{\ell}{n}:=\calE\cap\calS(\ell,n)\qand
\sy{n}:=\calE\cap\calS(n).
\]
It is straightforward to verify that $\syl{\ell}{n}$ and $\sy{n}$ are invariant under state space transformations. In addition, it follows from \eqref{e:data generated} that $\strue\in\syl{\lt}{\nt}\subseteq\sy{\nt}$.

\subsection{Informativity for system identification}
The available prior knowledge is formalized through a subclass of systems $\spk\subseteq\calS$ (with $\strue\in\spk$), that encapsulates what is known a priori about the true system \eqref{e:true sys}. 

\begin{definition}
\label{def:IWPKC}
The data $(\duf,\dyf)$ {\em are informative for system identification within $\spk$\/} if
\ben[label=\normalfont{(}\roman*\normalfont{)}, ref=(\roman*)]
\item\label{i:first} $\calE\cap\spk=\sy{\nt}\cap\spk$, and
\item\label{i:second} $\calE\cap\spk $  has isomorphism property.
\een
\end{definition}
Condition (i) states that all explaining systems in $\spk$ have $\nt$ states, and (ii)  that they are isomorphic to each other. 


In the rest of the paper, we assume that 
\emph{lower} and \emph{upper bounds} on the \emph{true lag} and \emph{state dimension} are given:
\beq\label{e:bounds on lt and nt}
L_-\leq\lt\leq L_+\leq N_+\text{ and } L_-\leq N_-\leq\nt\leq N_+.
\eeq
Of particular interest are those systems whose lags and state dimensions are within the given lower and upper bounds: 
\[
\calS_{[L_-,L_+],[N_-,N_+]}:=
\bigcup_{\scriptsize
\bmat
\ell\in[L_-,L_+]\\
n\in[N_-,N_+]
\emat
}
\calS(\ell,n).
\] 
The main results of this paper concern the informativity of the data for system identification within the prior knowledge class
\beq\label{e:spk choice min}
\spk=\calS_{[L_-,L_+],[N_-,N_+]}\cap\calM.
\eeq
In this case Definition~\ref{def:IWPKC} can be streamlined. Define
\[
\calE_{[L_-,L_+],[N_-,N_+]}:=
\bigcup_{\scriptsize
\bmat
\ell\in[L_-,L_+]\\
n\in[N_-,N_+]
\emat
}
\calE(\ell,n)\; ;
\]
since
\begin{gather}
\calE\cap\calS_{[L_-,L_+],[N_-,N_+]}=\calE_{[L_-,L_+],[N_-,N_+]}\notag\\
\calE(\nt)\cap\calS_{[L_-,L_+],[N_-,N_+]}=\calE(\nt)\; ,\label{e:trim 2}
\end{gather}
the proof of the following result is straightforward.
\begin{proposition}\label{p:streamline}
The data $(\duf,\dyf)$ are {\em  informative for system identification within $\calS_{[L_-,L_+],[N_-,N_+]}\cap\calM$\/} if and only if
\ben[label=\normalfont{(}\roman*\normalfont{)}, ref=(\roman*)]
\item\label{i:first-2} $\calE_{[L_-,L_+],[N_-,N_+]}\cap\calM=\sy{\nt}\cap\calM$, and
\item\label{i:second-2} $\sy{\nt}\cap\calM$ has isomorphism property.
\een
\end{proposition}

\begin{problem}\it
Given $T > 0$; $(\duf,\dyf)$ generated by a system \eqref{e:true sys};  $L_{-},L_{+},N_{-},N_{+}\in\N$ satisfying \eqref{e:bounds on lt and nt}; the model class $\spk$ in \eqref{e:spk choice min}, establish necessary and sufficient conditions for informativity for system identification within $\spk$.
\end{problem}

\subsection{Two relevant results}\label{sec:2relevantresults}
 We state \cite[Thm. 1]{Willems05} in our framework. 
\bprop\label{prop:wetal}
Suppose that $\lt\geq 1$ and 
\[
T\geq L_++N_+-1+(L_++N_+)m.
\]
If $ \duf $ is persistently exciting of order $ L_++N_+ $, that is,
\beq\label{e:inp pe}
\rank H_{L_++N_+-1}(\duf)=(L_++N_+)m,
\eeq
then
\[
\rank H_{L_++N_+-1}\left(\begin{bmatrix}
\duf\\\dyf
\end{bmatrix}\right)=(L_++N_+)m+\nt\; ,
\]
and the data $(\duf,\dyf)$ are informative for system identification within $\calS_{[1,L_+],[1,N_+]}\cap\calM$.
\eprop

We state \cite[Th.m 17]{Markovsky23}  in our framework. 
\bprop\label{prop:ivan-florian}
Suppose that $\lt\!\geq\! 1$ and
\[
T\geq \lt+(\lt+1)m+\nt.
\]
Then, the data $(\duf,\dyf)$ are informative for system identification within $\calS(\lt,\nt)\cap\calM$ if and only if 
\[
\rank H_{\lt}\lr{\bbm 
\duf\\
\dyf
\ebm}=(\lt+1)m+\nt\; .
\]
\eprop

%
%
%
%

\section{Main results}\label{s: summary main results}
Informativity for system identification within \eqref{e:spk choice min} is related to the rank of a Hankel matrix constructed from the data. We show that the depth of such Hankel matrix is determined by the data {and} by the   bounds $L_+$ and $N_+$. To state such condition we need to introduce some more notation and terminology.
\subsection{Data Hankel matrices}
For $k\in[0,T-1]$, we denote the \emph{Hankel matrix of depth $k$} constructed from the data $(\duf,\dyf)$ by $\hk{k}$ and the matrix obtained from $\hk{k}$ by deletion of its last $p$ rows by $\hhk{k}$:
\[
H_k:=\bbm
H_k(\duf)\\
H_k(\dyf)
\ebm
=
\begin{bsmallmatrix} 
u_{0}&\cdots&u_{T-k-1}\\
u_{1}&\cdots&u_{T-k}\\[-1mm]
\vdots&&\vdots\\
u_{k}&\cdots&u_{T-1}\\
y_{0}&\cdots&y_{T-k-1}\\
y_{1}&\cdots&y_{T-k}\\[-1mm]
\vdots&&\vdots\\
y_{k}&\cdots&y_{T-1}
\end{bsmallmatrix}=:
\begin{bsmallmatrix} 
&G_k&\\
y_{k}&\cdots&y_{T-1}
\end{bsmallmatrix}\; .
\]
Note that $\hk{k}\in \R^{(k+1)(m+p)\times (T-k)}$  and that $\hhk{k}$ 
is a ${((k+1)m+kp)\times (T-k)}$ matrix. We use the following sequence of integers  to state several intermediate results towards necessary and sufficient conditions for data informativity:  
\beq\label{e: rhok}
\delta_k:=
\begin{cases}
p&\text{if } k=-1\\
\rank H_k-\rank G_k&\text{if } k\in[0,T-1]\; ,
\end{cases}
\eeq
and note that 
\beq\label{e: rhok basic}
p\geq \delta_k\geq 0 \text{ for all } k\in[-1,T-1]\; .
\eeq
Throughout the paper, we assume that 
\[
\duf\neq 0_{m,T}\; ;
\]
then $1=\rank H_{T-1}=\rank G_{T-1}$ and hence 
\beq\label{e: rho T-1}
\delta_{T-1}=0\; .
\eeq

\subsection{Lags and state dimensions of explaining systems}
Let $q\in[0,T-1]$ be the smallest integer such that $\delta_q=0$: 
\beq\label{e:def q}
q:=\min\set{k\in[0,T-1]}{\delta_k=0}.
\eeq
Note that $q$ is well-defined due to \eqref{e: rho T-1}. In our first intermediate result we establish bounds on the lag and state dimension of {\em any\/} explaining system in terms of the  $\delta_k$'s. 
\begin{theorem}\label{t: n lower bound}
Suppose that $\sln\neq \emptyset$. 
The following statements hold:
\benabc
\item\label{t: n lower bound.1} If $T\geq \ell+1$, then $\ell\geq q$. 
\item\label{t: n lower bound.2} If $\ell\geq q$, then $n-\sum_{i=0}^{q} \delta_i\geq \ell-q$.
\end{enumerate}
\end{theorem}

The proof of this theorem is given in Section~\ref{proof of t: n lower bound}.

\subsection{Constructing an explaining system}

Our second intermediate result concerns whether {\em one\/} explaining system can be computed from the data. To answer such question, we introduce the notion of {\em state for data}. 

Let $\sa\in\sy{n}$. We say that $\dxf\in\R^{n\times(T+1)}$ is {\em a state for\/} $\sa$ if \eqref{e:explains} is satisfied. We say that
$\dxf$ is {\em a state for the data\/} $(\duf,\dyf)$ if it is a state for some $\sa\in\sy{n}$. 

From \eqref{e:explains}, it follows that $\dxf$ is a state for the data $(\duf,\dyf)$ if and only if
\beq\label{e:state output rsp}
\rs\bbm\dxfp\\\dyf\ebm\subseteq\rs\bbm\dxfm\\\duf\ebm.
\eeq
Consequently if {\em one\/} state sequence satisfying \eqref{e:state output rsp} is available, {\em one\/} explaining system can be computed by solving \eqref{e:explains}. In the next result we show that for {\em any\/} data $(\duf,\dyf)$ there {\em always\/} exists an explaining system with state dimension $\sum_{i=0}^{q} \delta_i$, and that \emph{all} such systems have lag $q$ (see \eqref{e:def q}).

\bthe\label{t:state construct}
$\emptyset\neq\sy{\sum_{i=0}^{q} \delta_i}=\syl{q}{\sum_{i=0}^{q} \delta_i}$.
\ethe
The  proof of Theorem \ref{t:state construct} in Section~\ref{sec:stateconstr} relies on a new iterative state-computation scheme from  input-output data. 

\subsection{The shortest lag and the minimum number of states}
We define  the {\em shortest lag\/} $\lm$ and the {\em  minimum number of states\/} $\nm$ required to explain the data $(\duf,\dyf)$:
\begin{align}
\lm&:=\min\{\ell\geq0 \mid \exists\, \sln\neq \emptyset \text{ for some } n\geq0 \}\label{e:def lm}\\
\nm&:=\min\{n\geq 0 \mid \sy{n}\neq \emptyset \}\; .\label{e:def nm}
\end{align}
$\lm$ and $\nm$ can be computed from the $\delta_k$'s and $q$ as follows.
\bthe\label{t:nmin} Define $q$ by \eqref{e:def q}; then $\lm=q$ and $\nm=\sum_{i=0}^{\lm}\delta_i$. Moreover, $\sy{\nm}=\syl{\lm}{\nm}\subset\calO$.
\ethe

\begin{proof}
From Theorem~\ref{t:state construct}, we see that $q\geq \lm$ and $\sum_{i=0}^q \delta_i\geq \nm$; to prove the first part of the theorem, it is enough to prove the reverse inequalities. The definition of $q$ implies that $T\geq q+1$; consequently $T\geq \lm+1$. Since $\calE(\lm,n)\nonempty$ for some $n$ due to the definition of $\lm$ in \eqref{e:def lm}, Theorem~\ref{t: n lower bound}.\ref{t: n lower bound.1} yields $\lm\geq q$ and hence $\lm=q$. The definition of $\nm$ in \eqref{e:def nm} implies that $\calE(\ell,\nm)\nonempty$ for some $\ell$. Then, Theorem~\ref{t: n lower bound}.\ref{t: n lower bound.2} implies that $\nm- \sum_{i=0}^{\lm} \delta_i\geq\ell-\lm\geq 0$. This proves $\nm\geq  \sum_{i=0}^{\lm} \delta_i$ and hence $\nm=\sum_{i=0}^{\lm} \delta_i $.

To prove the second part, use Theorem~\ref{t:state construct} to conclude that $\sy{\nm}=\syl{\lm}{\nm}$. It remains to prove that $\sy{\nm}\subset\calO$. Suppose on the contrary that $\sy{\nm}$ contains an unobservable system; a Kalman decomposition  yields an explaining system with strictly lower state dimension. This contradicts the definition of $\nm$ in \eqref{e:def nm}. Consequently, all systems in $\sy{\nm}$ are observable: $\sy{\nm}\subset\calO$.\end{proof} 

\subsection{Sharpening the upper bound on the true lag}
Assume that $\sy{\ell,n}\neq \emptyset$; from Theorems~\ref{t: n lower bound}.\ref{t: n lower bound.2} and \ref{t:nmin} we obtain the  im\-me\-diate but crucial inequality
\beq\label{e: l bound with n lm nm}
\ell\leq n-\nm+\lm\; .
\eeq
Such inequality implies that the integer $\ld$ defined by 
\[
\ld:=N_+-\nm+\lm
\]
is an upper bound, purely determined by the data and $N_+$, for the lag of {\em every\/} explaining system with at most $N_+$ states. Such upper bound  yields a sharper  upper bound on the lag than  $L_+$; we can replace the latter by the {\em actual\/} upper bound
\[
\la:=\min(L_+,\ld)
\]  
since it follows from \eqref{e: l bound with n lm nm} that
\beq\label{e:trimming wrt la}
\calE_{[L_-,L_+],[N_-,N_+]}=\calE_{[L_-,\la],[N_-,N_+]}.
\eeq

\subsection{Data informativity for system identification}

The following  is the main result of this paper;  sufficiency  is proved in Section \ref{sec:t:mainsuff} and  necessity  in Section \ref{sec:t:mainnec}. 
\bthe\label{t:main}
The data $(\duf,\dyf)$ are informative for system identification within $\calS_{[L_-,L_+],[N_-,N_+]}\cap\calM$ if and only if the following conditions hold:
\bse\label{e:nec suff conditions}
\begin{gather}
\lm \geq L_-\label{e:lb l}\\
\nm\geq N_-\label{e:lb n}\\
T\geq \la+(\la+1)m+\nm\label{e:lb T}\\
\rank H_{\la}=(\la+1)m+\nm\; . \label{e:main rank condition}
\end{gather}
\ese
Moreover, if the conditions \eqref{e:nec suff conditions} are satisfied then 
\bse\label{e:extras}
\begin{gather}
\lt=\lm\label{e:extra-lm}\\
\nt=\nm\label{e:extra-nm}\\
\calE\cap\calS_{[L_-,L_+],[N_-,N_+]}\cap\calM=\sy{\nm}\; .\label{e:extra-main}
\end{gather}
\ese
\ethe 
Equations \eqref{e:nec suff conditions} are \emph{truly data-based} necessary and sufficient condition for informativity for system identification: $\lm$ and $\nm$ are computed directly from the data via Theorem~\ref{t:nmin}.

We conclude this section with some remarks on the consequences of Theorem~\ref{t:main} and its relation to other results.

\brem\rm 
It is not surprising that informativity for system identification involves a  rank condition on some  Hankel matrix of the data. Theorem~\ref{t:main} makes explicit that the depth of such Hankel matrix  depends {\em not only\/} on the prior knowledge of the system {\em but also\/} on the given data.\hfill \qed
\erem

\begin{remark}\label{rem:one upper bound}\rm 
Theorem~\ref{t:main} is applicable even if one of the upper bounds is unknown. Indeed, if only the upper bound $L_+$ on the lag is known, but an upper bound $N_{+}$ on the state dimension is not, one can fix $N_+=pL_+$ since $n\leq p\ell$ for any observable system in $\calS(\ell,n)$. For such choice it holds that
\[
N_+-\nm+\lm\geq pL_+-p\lm+\lm\geq L_+\; ;
\]
 the first inequality follows from the fact that $\nm\leq p\lm$ as $\calE(\lm,\nm)\subset\calO$ (see Theorem~\ref{t:nmin}) and the second one from $L_+\geq \lt\geq \lm.$ As such, if $ N_+=pL_+ $ then $\la=L_+$.

Conversely, if an upper bound $N_+$ on the state dimension is known but an upper bound $L_{+}$ on the lag is not, one can fix $L_+=N_+$ since the lag of a system cannot exceed its state dimension. For such choice of $L_+$ it holds that
\[
L_+\geq L_+-\nm+\lm=N_+-\nm+\lm
\]
where the inequality follows from the fact that $\nm\geq\lm$. Therefore, if $L_+=N_+$ then $ \la=\ld $.\hfill \qed
\end{remark}

\brem\rm 
If $L_-=L_+=\lt$ and $L_+=N_+=\nt$ then it is straightforward to see that Proposition~\ref{prop:ivan-florian} is a special case of Theorem~\ref{t:main} (see Section \ref{sec:2relevantresults}). It is less straightforward to prove that the rank condition \eqref{e:inp pe} in Proposition~\ref{prop:wetal} implies the conditions in Theorem~\ref{t:main}; we do this in Section~\ref{sec:prop to main thm}.  \hfill \qed
\erem

\brem\label{rem:data length}\rm 
Proposition~\ref{prop:wetal} can be applied only if the data length is {\em at least\/} $L_++N_++(L_++N_+)m-1$ whereas Theorem~\ref{t:main} can be applied if the data length is {\em at least\/} $\la+(\la+1)m+\nt$. As illustrated in Section~\ref{sec:running example}, the difference between these  lengths can be significantly large.\hfill \qed
\erem

\brem\rm
Theorem~\ref{t:main} is important for online experiment design. Assume that $\lt\geq 1$ and that only an upper bound $L_+$ is known; the procedure in \cite[Thm. 3]{Waarde21}  constructs a sequence $\duf$ with $T=(L_++1)m+L_++\nt$ such that the  data $(\duf,\dyf)$ generated by the true system are informative for system identification in $\calS_{[1,L_+],[1,pL_+]}\cap\calM$. Surprisingly, the procedure does {\em not\/} require  knowledge of $\nt$. For this case $\la=L_+$ (see Remark~\ref{rem:one upper bound}). If $(\duf,\dyf)$ are informative for system identification in $\calS_{[1,L_+],[1,pL_+]}\cap\calM$, then from \eqref{e:main rank condition} and \eqref{e:extra-nm} that $T\geq (L_++1)m+L_++\nt$. Thus the experiment design procedure in \cite[Thm. 3]{Waarde21} is {\em minimal\/} in the number of samples required. \hfill \qed
\erem

\brem\rm
The constructive proof of Theorem~\ref{t:state construct} and \eqref{e:extra-main} can be used to compute from  informative data an explaining system isomorphic to the true system. We show this in Example \ref{ex:state construct example} after the proof of Theorem~\ref{t:state construct}.\hfill \qed
\erem

\brem\rm 
The lower bounds $L_{-}$, $N_{-}$ are inconsequential for data informativity: if the data are informative within $\calS_{[L_-,L_+],[N_-,N_+]}\cap\calM$ then  necessarily $\lt=\lm$ and $ \nt=\nm$ (see \eqref{e:extra-lm}, \eqref{e:extra-nm}) and consequently the  data are also informative in $\calS_{[0,L_+],[0,N_+]}\cap\calM$. The converse holds since $\calS_{[L_-,L_+],[N_-,N_+]}\cap\calM\subseteq \calS_{[0,L_+],[0,N_+]}\cap\calM$. \hfill \qed 
\erem

Before proving Theorems \ref{t: n lower bound}, \ref{t:state construct} and \ref{t:main} we illustrate them with an example.

\section{Illustrative example}\label{sec:running example}
Consider a system \eqref{e:true sys} where $\nt=3$, $m=2$, $p=2$, and
\[
\struebig=
\left[
\begin{array}{ccc|cc}
0 & 1 & 0 & 1 & 0\\
0 & 0 & 1 & 0 & 1\\
0 & 0 & 0 & 0 & 1\\
\hline
1 & 0 & 0 & 1 & 0\\
0 & 1 & 0 & 0 & 0
\end{array}
\right].
\]
Note that $\lt=2$. Consider the input-output data
\[
\left[\begin{array}{c}
\du{0}{13}\\[1mm]
\hline
\dy{0}{13}
\end{array}\right]
=
\left[\begin{array}{cccccccccccccc}
1\! &\! 1\! &\! 1\! &\! 0\! &\! 0\! &\! 0\! &\! 0\! &\! 1\! &\! 0\! &\! 0\! &\! 0\! &\! 0\! &\! 1\! &\! 1\\
0\! &\! 0\! &\! 0\! &\! 1\! &\! 1\! &\! 1\! &\! 1\! &\! 1\! &\! 1\! &\! 1\! &\! 1\! &\! 1\! &\! 0\! &\! 0\\
\hline
2 \!&\!3 \!&\!2 \!&\!1 \!&\!0 \!&\!1 \!&\!2 \!&\!3 \!&\!3 \!&\!2 \!&\!2 \!&\!2 \!&\!3 \!&\!4\\
1 \!&\!0 \!&\!0 \!&\!0 \!&\!1 \!&\!2 \!&\!2 \!&\!2 \!&\!2 \!&\!2 \!&\!2 \!&\!2 \!&\!2 \!&\!1
\end{array}\right].
\]
One can verify that \eqref{e:data generated} is satisfied with the state data
\[
\dx{0}{14}=
\left[\begin{array}{ccccccccccccccc}
1 \!&\! 2 \!&\! 1 \!&\! 1 \!&\! 0 \!&\! 1 \!&\! 2 \!&\! 2 \!&\! 3 \!&\! 2 \!&\! 2 \!&\! 2 \!&\! 2 \!&\! 3 \!&\! 2\\
1 \!&\! 0 \!&\! 0 \!&\! 0 \!&\! 1 \!&\! 2 \!&\! 2 \!&\! 2 \!&\! 2 \!&\! 2 \!&\! 2 \!&\! 2 \!&\! 2 \!&\! 1 \!&\! 0\\
0 \!&\! 0 \!&\! 0 \!&\! 0 \!&\! 1 \!&\! 1 \!&\! 1 \!&\! 1 \!&\! 1 \!&\! 1 \!&\! 1 \!&\! 1 \!&\! 1 \!&\! 0 \!&\! 0
\end{array}\right].
\]
Table~\ref{tab:integers} presents the values of $\delta_k$ integers, $\lm$, and $\nm$ for different choices of $T$. The values of $\delta_k$ not explicitly indicated are zero: $\delta_k=0$ for every $T\in[4,14]$ and $k\in[3,T-1]$.
\renewcommand{\arraystretch}{1.2}
\begin{table}[!h]
\begin{center}
\begin{tabular}{c|cccc}
$ T $ & $ 1 $ & $ 2 $ & $ [3,5] $ & $ [6,14] $\\
\hline
$ \delta_{-1} $& $ 2 $ & $ 2 $ & $ 2 $ & $ 2 $\\
$ \delta_{0} $ & $ 0 $ & $ 1 $ & $ 2 $ & $ 2 $\\
$ \delta_{1} $ &       & $ 0 $ & $ 0 $ & $ 1 $\\
$ \delta_{2} $ &       &       & $ 0 $ & $ 0 $\\
\hline
$ \lm $ & 0 & 1 & 1 & 2\\
$ \nm $ & 0 & 1 & 2 & 3
\end{tabular}
\end{center}
\caption{$\delta_k$ integers, $\lm$, and $\nm$ for different choices of $T$}\label{tab:integers}
\end{table}
\vspace*{-3mm}
If $L_+=\lt$ and $N_+=\nt$, a necessary condition for informativity is that $T\geq (\lt+1)m+\lt+\nt$ (see \eqref{e:lb T}). It follows that if $T<11$ then $(\duf,\dyf)$ is not informative for every $L_+$ and $N_+$. The symbol `\checkmark' in Table~\ref{tab:data info} denotes data informativity for different values of $L_+$, $N_+$, and $T$, as  inferred from the fundamental lemma (Proposition~\ref{prop:wetal}) and Theorem~\ref{t:main}.  Proposition~\ref{prop:wetal} requires significantly more samples than Theorem~\ref{t:main} (see Remark~\ref{rem:data length}): the data $(u_{[0,13]},y_{[0,13]})$ are informative  for the case $L_+=N_+=4$ (see Theorem~\ref{t:main}). To infer informativity via the fundamental lemma, one would need at least $L_++N_++(L_++N_+)m-1=23$ samples.
\setlength{\tabcolsep}{0.4em}
\renewcommand{\arraystretch}{1.2}
\begin{table}[!h]
\begin{center}
\begin{tabular}{cc|cc|cccc|cccc|}
\multicolumn{4}{c}{} & \multicolumn{4}{|c|}{Proposition~\ref{prop:wetal}} & \multicolumn{4}{c|}{Theorem~\ref{t:main}}\\
\cline{5-12}
\multicolumn{4}{c}{} & \multicolumn{4}{|c|}{$ T $ values} & \multicolumn{4}{c|}{$ T $ values}\\
\cline{5-12}
$L_+$ & $N_+$ & $\ld$ & $\la$ & $11$ & $12$ & $13$ & $14$ & $11$ & $12$ & $13$ & $14$\\
\hline
$2$ & $3$ & $2$ & $2$ & & & & \checkmark&\checkmark & \checkmark & \checkmark & \checkmark\\
$2$ & $4$ & $3$ & $2$ & & & & &\checkmark & \checkmark & \checkmark & \checkmark\\
$2$ & $5$ & $4$ & $2$ & & & & &\checkmark & \checkmark & \checkmark & \checkmark\\
$2$ & $6$ & $5$ & $2$ & & & & &\checkmark & \checkmark & \checkmark & \checkmark\\
$3$ & $3$ & $2$ & $2$ & & & & &\checkmark & \checkmark & \checkmark & \checkmark\\
$3$ & $4$ & $3$ & $3$ & & & &&  &  &  & \checkmark\\
$3$ & $5$ & $4$ & $3$ & & & &&  &  &  & \checkmark\\
$3$ & $6$ & $5$ & $3$ & & & &&  &  &  & \checkmark\\
$4$ & $4$ & $3$ & $3$ & & & &&  &  &  & \checkmark
\end{tabular}
\end{center}
\caption{Informativity of the data for different bounds and $T$}\label{tab:data info}
\end{table}

\renewcommand{\arraystretch}{1}

\section{Lag structures of explaining systems}\label{sec:lag struct}
To prove Theorem~\ref{t: n lower bound}, we first define the lag structure and relate it to the integers $\delta_k$.

\subsection{The lag structure of a system}

Let $\sa\in \calS(\ell,n)$. For $k\geq -1$, define $\Omega_k$ by \eqref{e:def omega-k} and   
\[
\rho_k:=\begin{cases} p&\text{if } k=-1\\
\rank \Omega_k-\rank \Omega_{k-1}&\text{if } k\geq 0. \end{cases}
\]
We refer to the sequence $(\rho_k)_{k\in\N}$ as the {\em lag structure\/} of the system $\sa$. The $\rho_k$'s  are related to a \emph{specific} system; if necessary to resolve ambiguities, we use the notation $\rho_k(C,A)$. 
\begin{lemma}\label{p: sum rk}
Let $\sa\in \calS(\ell,n)$ and $(\rho_k)_{k\in\N}$ be its lag structure. The following statements hold:
\benabc
\item\label{p: sum rk.aa} $p\geq \rho_k\geq 0$ for all $k\geq -1$
\item\label{p: sum rk.bb} $\rho_{\ell-1}\geq 1$ and $\rho_k=0$ for all $k\geq\ell$,
\item\label{p: sum rk.ee} $n\geq \sum_{i=0}^{\ell} \rho_i$; if $(C,A)$ is observable then equality holds.
\item\label{p: sum rk.cc} $\rho_{k}\geq\rho_{k+1}$ for all $k\geq 0$.
\end{enumerate}
\end{lemma}

\begin{proof}
The statements \ref{p: sum rk.aa} and \ref{p: sum rk.bb} readily follow from the definitions of $\rho_k$ and the lag. To prove \ref{p: sum rk.ee}, note that $\rank\Omega_\ell=\sum_{i=0}^{\ell} \rho_i$. This proves \ref{p: sum rk.ee} since $n\geq\rank\Omega_\ell$, and equality holds if $(C,A)$ is observable.

To prove \ref{p: sum rk.cc}, we first give an alternative characterization of the integers $\rho_k$. Let $k\geq 0$. Note that $\rs\Omega_{k}=\rs\Omega_{k-1}+\rs CA^k$ since $\Omega_{k}=\bbm \Omega_{k-1}\\CA^k\ebm$. As such, we have $\rank\Omega_{k}=\rank\Omega_{k-1}+\rank CA^k-\dim V_k$ where $V_k:=\rs\Omega_{k-1}\cap\rs CA^k$. This leads to the following alternative characterization for $\rho_k$:
\beq\label{e:alt delta}
\rho_k=\rank CA^k-\dim V_k\; .
\eeq
Now let $k\geq 0$ and observe that $CA^{k+1}=CA^{k}A$. Apply the rank-nullity theorem and obtain
\beq\label{e:cak to cak-1}
\rank CA^{k}=\rank CA^{k+1}+\dim W_{k}\; ,
\eeq
where $W_{k}:=\rs CA^{k}\cap \lk A$. By combining \eqref{e:alt delta} and \eqref{e:cak to cak-1}, we obtain $\rho_{k}-\rho_{k+1}=\dim W_{k}+\dim V_{k+1}-\dim V_{k}$. Let $Z_{k}$ be a subspace such that $V_{k}=(V_{k}\cap W_{k})\oplus Z_{k}$. Then, we have 
\begin{align}
\rho_{k}-\rho_{k+1}\geq \dim V_{k+1}-\dim Z_{k}.\label{e:rk ito v and w and z}
\end{align}
Let $d=\dim Z_{k}$ and $\eta_i\in\R^{1\times n}$ with $i\in[1,d]$ be a basis for $Z_{k}$. From the definition of $Z_{k}$, it readily follows that $\eta_iA$ are linearly independent and $\eta_i A\in \rs \Omega_{k-1}A\cap\rs CA^{k+1}\subseteq \rs \Omega_{k}\cap\rs CA^{k+1}=V_{k+1}$. Therefore, $\dim Z_{k}\leq \dim V_{k+1}$. Hence, it follows from \eqref{e:rk ito v and w and z} that $\rho_{k}\geq\rho_{k+1}$ for all $k\geq 0$.\end{proof}

\subsection{Lag structures and $\delta_k$ integers}
Let $\sa\in \calS(n)$. For $k\geq -1$, define the \emph{$k$-th controllability matrix}, and the \emph{$k$-th system matrix}, respectively, by 
\begin{align}
\Gamma_k&:=\begin{cases}
0_{n,0}&\mbox{ if } k=-1\\
\bbm A^{k}B&\Gamma_{k-1}\ebm&\mbox{ if } k\geq 0 
\end{cases} \\
\Theta_k&:=\begin{cases}
0_{0,0}&\mbox{ if } k=-1\\
\bbm \Theta_{k-1} & 0\\C\Gamma_{k-1}&D\ebm&\mbox{ if } k\geq 0\; . 
\end{cases}
\end{align}
$\Theta_{k}$ is the Toeplitz matrix of the first $k+1$ Markov parameters of \eqref{eq: generic linear system}. Given a state sequence $\dxf$ for an explaining system $\sa\in\calE(n)$, the data Hankel matrices $H_k$ and $G_k$ are related to the observability and system matrices by:
\bse
\label{e:hankel data from system and state}
\begin{align}
H_k&=\bbm
H_k(\duf)\\
H_k(\dyf)
\ebm=\Phi_k
\bbm
\dx{0}{T-k-1}\\
H_k(\duf)
\ebm\label{e:hk phik jk}\\
G_k&=\bbm H_k(\duf)\\
H_{k-1}(\dy{0}{T-2})\ebm=
\Psi_k
\bbm
\dx{0}{T-k-1}\\
H_k(\duf)
\ebm\label{e:gk psik jk}
\end{align}
\ese
where
\beq\label{e:Phi and Psi}
\Phi_k\!:=\bbm
0&I\\
\Omega_k&\Theta_k 
\ebm\text{ and }
\Psi_k\!:=
\bbm
0&I&0\\
0&0&I\\
\Omega_{k-1}&\Theta_{k-1} &0_{kp,m}
\ebm\!\!\!.
\eeq
We now relate the integers $\delta_k$ (defined in terms of the \emph{data matrices} $\hk{k}$ and $\hhk{k}$ only) and the integers $\rho_k$ (defined by the matrices $C$ and $A$ of a \emph{specific explaining system}). 
\blem\label{p:sum sk sum rk}
Let $(\rho_k)_{k\in\N}$ be the lag structure of an explaining system. For every $k\in[0,T-1]$, $\rho_k\geq\delta_k$. 
\elem

\begin{proof}
Let $k\in[0,T-1]$. Define $J_k:=\begin{bsmallmatrix}
\dx{0}{T-k-1}\\
H_k(\duf)
\end{bsmallmatrix}$;  by applying the rank-nullity theorem to 
\eqref{e:hankel data from system and state}, we see that
\begin{align}
\rank H_k&=\rank \Phi_k-\dim(\rs \Phi_k\cap \lk J_k)\label{e:rank hk-char}\\
\rank G_k&=\rank \Psi_k-\dim(\rs \Psi_k\cap \lk J_k)\label{e:rank gk-char}\; .
\end{align}
Now $\rank\Phi_k-\rank\Psi_k=\rank\Omega_k-\rank \Omega_{k-1}$ and $\rs\Psi_k\subseteq\rs\Phi_k$ from  \eqref{e:Phi and Psi}; subtract \eqref{e:rank gk-char} from \eqref{e:rank hk-char} to obtain $\rho_k\geq\delta_k$.
\end{proof}

We are ready to prove Theorem~\ref{t: n lower bound}.

\subsection{Proof of Theorem~\ref{t: n lower bound}}\label{proof of t: n lower bound}
Let $\sa\in\sln$ and let $(\rho_k)_{k\in\N}$ be the lag structure of the system $\sa$. To prove \ref{t: n lower bound.1}, note that $\rho_\ell=0$ due to Lemma~\ref{p: sum rk}.\ref{p: sum rk.bb}. Since $T-1\geq \ell$ by hypothesis, Lemma~\ref{p:sum sk sum rk} and \eqref{e: rhok basic} imply that $\delta_\ell=0$. Then $\ell\geq q$ from  \eqref{e:def q}. 

To prove \ref{t: n lower bound.2}, note that
\beq\label{e:n- something}
n-\sum_{i=0}^{q} \delta_i\geq \sum_{i=0}^{\ell} \rho_i-\sum_{i=0}^{q} \delta_i
\eeq
due to Lemma~\ref{p: sum rk}.\ref{p: sum rk.ee}. Therefore, if $\ell=q$ then \ref{t: n lower bound.2} readily follows from Lemma~\ref{p:sum sk sum rk}. Suppose that $\ell>q$. Note that
\[
\sum_{i=0}^{\ell} \rho_i-\sum_{i=0}^{q} \delta_i\geq \sum_{i=q}^{\ell-1}\rho_i
\geq \ell-q
\]
where the first inequality follows from Lemma~\ref{p: sum rk}.\ref{p: sum rk.bb} and Lemma~\ref{p:sum sk sum rk}, and the second from the statements \ref{p: sum rk.bb} and \ref{p: sum rk.cc} of Lemma~\ref{p: sum rk}. Consequently, \ref{t: n lower bound.2} follows from \eqref{e:n- something}.\EP

\section{State construction}\label{sec:construct state}
In this section, we present a new iterative state construction procedure from  data instrumental in proving Theorem~\ref{t:state construct}. 

\subsection{On the left kernels of data Hankel matrices}
The  $\delta_k$'s defined by \eqref{e: rhok} are strongly related to certain subspaces of the left kernels of data Hankel matrices. To show this, we first introduce a  {\em shift\/} operator on subspaces. Let $\calV\subseteq\R^{1\times \kappa(m+p)}$ be a subspace where $\kappa\in\N$. Define $\sigma\calV$ as the subspace of all vectors of the form
$
\begin{bmatrix} 0_{1\times m}&v_1&0_{1\times p}&v_2\end{bmatrix}
$
where $v_1\in\R^{1\times\kappa m}$ and $v_2\in\R^{1\times\kappa p}$ satisfy
$
\begin{bmatrix} v_1&v_2\end{bmatrix}\in \mathcal{V}.
$
By convention, $\sigma^0\calV:=\calV$. Moreover, $\sigma^k\calV=\sigma (\sigma^{k-1}\calV)$ for $k\geq 1$. The definitions of $H_k$ and $G_k$ imply that 
\beq\label{e:lk G in lk H}
\lk G_k\times{\bz}_p\subseteq\lk H_k
\eeq
for all $k\in[0,T-1]$. From the definition of $\sigma$ and the Hankel structure it follows that if $T\geq 2$, then for all $k\in[0,T-2]$
\begin{align}
\sigma\lk\hk{k}&\subseteq\lk\hk{k+1}\label{e:hk in hk+1.1}\\
\sigma (\lk \hhk{k}\times{\bz}_p)&\subseteq \lk \hhk{k+1}\times{\bz}_p\label{e:hk in hk+1.2}\\
\!\!\!\!\sigma\lk H_k\cap (\lk G_{k+1}\times {\bz}_p)&=\sigma(\lk G_{k}\times {\bz}_p).\label{e:hk in hk+1.3}
\end{align}
The following result shows that $\lk \hk{k}$ can be written into a direct sum of $\lk \hhk{k} \times \{0\}^p$ and shifts of certain subspaces. 
\begin{lemma}\label{l:si subspaces exist}
For $k\in[0,T-1]$, there exist subspaces $\calS_{k}\subseteq\R^{1\times (k+1)(m+p)}$ satisfying 
\beq\label{e:def Si}
\lk \hk{k}=
\lr{\bigoplus_{i=0}^k\sigma^{k-i}\calS_i}
\oplus
\lr{\lk\hhk{k}\times{\bz}_p}
\eeq
and $\dim \calS_k =\delta_{k-1}-\delta_k$.
\end{lemma}

\begin{proof} To prove the existence of subspaces satisfying \eqref{e:def Si} we use induction on $k$. From \eqref{e:lk G in lk H} with $k=0$, we see that there exists a subspace $\calS_0\subseteq\R^{1\times (m+p)}$ such that $\lk\hk{0}=\calS_0\oplus\lr{\lk\hhk{0}\times{\bz}_p}$. If $T=1$, there is nothing more to prove. Suppose that $T\geq 2$. Let $k\in[0,T-2]$ and assume that there exist subspaces $\calS_{i}\subseteq\R^{1\times (i+1)(m+p)}$ with $i\in[0,k]$ satisfying
\beq\label{e:def Sl}
\lk \hk{k}=
\lr{\bigoplus_{i=0}^k\sigma^{k-i}\calS_i}
\oplus
\lr{\lk\hhk{k}\times{\bz}_p}.
\eeq
From \eqref{e:lk G in lk H}, \eqref{e:hk in hk+1.1} conclude that $\sigma\lk\hk{k}+(\lk \hhk{k+1}\times{\bz}_p)\subseteq \lk H_{k+1}$:  there exists $\calS_{k+1}\subseteq\R^{1\times (k+2)(m+p)}$ such that 
\beq\label{e:Sk+1 debute}
\lk\hk{k+1}\!=\!\calS_{k+1}\oplus \lr{\sigma\lk\hk{k}+(\lk \hhk{k+1}\times{\bz}_p)}.
\eeq
Given subspaces $\calV_1,\calV_2,\calV_3$ with $\calV_1\cap\calV_2=\zset$ and $(\calV_1+\calV_2)\cap\calV_3=\calV_2$, it holds that $\calV_1+\calV_2+\calV_3=\calV_1\oplus\calV_3$. Take $\calV_1=\sigma(\bigoplus_{i=0}^k\sigma^{k-i}\calS_i)$, $\calV_2=\sigma(\lk\hhk{k}\times{\bz}_p)$, and $\calV_3=\lk\hhk{k+1}\times{\bz}_p$. Note that $\calV_1\cap\calV_2=\zset$ due to \eqref{e:def Sl} and $(\calV_1+\calV_2)\cap\calV_3=\calV_2$ due to \eqref{e:hk in hk+1.3}. Therefore, we see from \eqref{e:def Sl} and \eqref{e:Sk+1 debute} that
$
\lk \hk{k+1}=
\lr{\bigoplus_{i=0}^{k+1}\sigma^{k+1-i}\calS_i}
\oplus
\lr{\lk\hhk{k+1}\times{\bz}_p}.
$
This proves that there exist subspaces $\calS_k$ such that \eqref{e:def Si} holds. To complete the proof, it remains to show that $\dim\calS_k=\delta_{k-1}-\delta_k$ for $k\in[0,T-1]$. Let $k\in[0,T-1]$ and observe that \eqref{e:def Si} yields $
\dim\lk H_k=\sum_{i=0}^k\dim\calS_i+\dim\lk G_k.$ From the rank-nullity theorem, we conclude that $
\sum_{i=0}^k\dim\calS_i=\delta_{-1}-\delta_{k}.$ Therefore, $\dim \calS_0=\delta_{-1}-\delta_{0}$ and $
\dim\calS_k=\sum_{i=0}^k\dim\calS_i-\sum_{i=0}^{k-1}\dim\calS_i=
\delta_{k-1}-\delta_{k}$ for all $k\in[1,T-1]$.
\end{proof}

\subsection{Proof of Theorem~\ref{t:state construct}}\label{sec:stateconstr}
We first prove that $\sy{\sum_{i=0}^{q} \delta_i}\nonempty$ by constructing a state $x\in\R^{\sum_{i=0}^{q} \delta_i\times (T+1)}$ for the data $(\duf,\dyf)$. 

Let subspaces $\calS_k\subseteq \R^{1\times (k+1)(m+p)}$ with $k\in[0,T-1]$ be as in Lemma~\ref{l:si subspaces exist}. Also, denote $s_k:=\dim\calS_k$. For $i\in[0,q]$, let  $Q_{i,j}\in\R^{s_i\times m}$ and $P_{i,j}\in\R^{s_i\times p}$ with $j\in[0,i]$ be such that the rows of the matrix
$
R_i:=\begin{bmatrix}
Q_{i,0} & Q_{i,1} & \cdots&Q_{i,i} &P_{i,0} & P_{i,1} &\cdots & P_{i,i}
\end{bmatrix}
$
form a basis for $\calS_i$. Note that $R_i\in\R^{s_i\times (i+1)(m+p)}$. 
    
Since $\calS_i\subseteq\lk H_i$ due to \eqref{e:def Si},  $R_i H_i=0$ and hence 
\beq\label{e:def q and p implies}
\sum_{j=0}^{i}Q_{i,j}\du{j}{T-1-i+j}+P_{i,j}\dy{j}{T-1-i+j}=0 .
\eeq
Due to Lemma~\ref{l:si subspaces exist}, 
$\sum_{i=0}^{q} s_i=\delta_{-1}-\delta_{q}=p$ since $\delta_{-1}=p$ and $\delta_q=0$ by definition. 
Therefore, the matrix  
$
\Pi:=\col( 
P_{0,0},
P_{1,1},\ldots,P_{q,q})$
is $p\times p$.

We claim that $\Pi$ is nonsingular. To see this, let $\eta\in\R^{1\times p}$ be such that $\eta \Pi=0$. Define
$$
R:=\begin{bsmallmatrix}
0 &0 &\cdots& 0 & Q_{0,0}&0&0&\cdots &0& P_{0,0}\\
 0 & 0 &\cdots& Q_{1,0} &Q_{1,1}&0&0&\cdots&   P_{1,0}& P_{1,1}\\
 \vdots & \vdots &  & \vdots & \vdots & \vdots  & & \vdots& \vdots\\
 Q_{q,0} & Q_{q,1}&\cdots & Q_{q,q-1} & Q_{q,q}  & P_{q,0} & P_{q,1} &\cdots&P_{q,q-1}& P_{q,q}
\end{bsmallmatrix}
$$
and observe that $\Pi$ is the last block-column of $R$. 

From the definition of $Q_{i,j}$ and $P_{i,j}$, it is straightforward to verify that the rows of $R$ form a basis for the subspace $\bigoplus_{i=0}^{q}\sigma^{(q-i)}\calS_i$. Then, $\eta R\in\bigoplus_{i=0}^{q}\sigma^{(q-i)}\calS_i$. This means that $RH_q=0$. Since $\eta \Pi=0$, the last $p$ entries of $\eta R$ are zero. Therefore, $\eta R\in \lr{\lk\hhk{q}\times{\bz}_p}$ and hence 
$\eta R\in\lr{\bigoplus_{i=0}^{q}\sigma^{(q-i)}\calS_i}
\cap
\lr{\lk\hhk{q}\times{\bz}_p}$. Since $\lr{\bigoplus_{i=0}^{q}\sigma^{(q-i)}\calS_i}
\cap
\lr{\lk\hhk{q}\times{\bz}_p}=\{0\}$ due to \eqref{e:def Si}, we conclude that $\eta R=0$. Since the rows of $R$ are linearly independent, we see that $\eta=0$ and thus
$\Pi\in\Rpp$ is nonsingular. 

We now distinguish two cases: $q=0$ and $q\geq1$.

For the case $q=0$, we have $s_0=\delta_{-1}-\delta_0=p$ and $\Pi=P_{0,0}$. It follows from \eqref{e:def q and p implies} with $i=0$ and the nonsingularity of $P_{0,0}$ that $\dyf=-P_{0,0}\inv Q_{0,0}\duf$. Consequently, the memoryless model associated with $-P_{0,0}\inv Q_{0,0}$ explains the data: $-P_{0,0}\inv Q_{0,0}\in\sy{0}\nonempty$. Together with $\sum_{i=0}^{q} \delta_i=0$, this proves the claim for the case $q=0$.

For the case $q\geq 1$, we first construct some auxiliary sequences from the data and then we show that such sequences can be used to compute a state for the data $(\duf,\dyf)$. 

Let  $i\in[1,q]$ and for $k\in[1,i]$ define $x_0^{i,k}\in\R^{s_i}$ by
\beq\label{e:i k 0}
x_0^{i,k}:=\sum_{j=k}^{i}(Q_{i,j}u_{j-k}+P_{i,j}y_{j-k}).
\eeq
Define $\dxfpp^{i,k}\in\R^{s_i\times T}$ by 
\beq\label{e:i 1 1-T}
\dxfpp^{i,1}:=-Q_{i,0}\du{0}{T-1}-P_{i,0}\dy{0}{T-1}
\eeq
for $k=1$ and by 
\beq\label{e:i k 1-T}
\dxfpp^{i,k}:=\dxfm^{i,k-1}-Q_{i,k-1}\du{0}{T-1}-P_{i,k-1}\dy{0}{T-1} 
\eeq
for $k\in[2,i]$. Finally, define  
$
\dx{0}{T}^i:=\col(\dx{0}{T}^{i,1},\dx{0}{T}^{i,2},\ldots,\dx{0}{T}^{i,i}) \in\R^{is_i\times(T+1)}
$
and
$
\dx{0}{T}:=\col(\dx{0}{T}^{1},\dx{0}{T}^{2},\ldots,\dx{0}{T}^{q})\in \R^{\left(\sum_{i=1}^{q}is_i \right)\times(T+1)} .
$

We claim that
\beq\label{e:i i 0-T}
\dx{0}{T-1}^{i,i}=Q_{i,i}\du{0}{T-1}+P_{i,i}\dy{0}{T-1}.
\eeq
To verify this, note first that due to \eqref{e:i k 0}
\beq\label{e:i i 0}
x^{i,i}_0=Q_{i,i}u_0+P_{i,i}y_0\; .
\eeq
Consider the case $i=1$; then 
$
\dx{1}{T-1}^{1,1}\overset{\eqref{e:i 1 1-T}}{=}-Q_{1,0}\du{0}{T-2}-P_{1,0}\dy{0}{T-2}.
$
Now 
$
\dx{1}{T-1}^{1,1}\overset{\eqref{e:def q and p implies}}{=}Q_{1,1}\du{1}{T-1}+P_{1,1}\dy{1}{T-1}
$; using  \eqref{e:i i 0}, we conclude that \eqref{e:i i 0-T} holds for $i=1$. 

To prove \eqref{e:i i 0-T} for the case $i>1$, let $\alpha\in[1,T-1]$. We first consider the case $\alpha\in[1,i-1]$. Note that  
\begin{align}
&x^{i,i}_{\alpha}-x^{i,i-\alpha}_{0}=\sum_{j=i-\alpha+1}^i(x^{i,j}_{\alpha-i+j}-x^{i,j-1}_{\alpha-i+j-1})\notag\\
&\overset{\eqref{e:i k 1-T}}{=}-\sum_{j=i-\alpha+1}^i
(Q_{i,j-1}u_{\alpha-i+j-1}+P_{i,j-1}y_{\alpha-i+j-1}).\label{e:x alpha i i - x zero i i - alpha}
\end{align}
Now 
$
x^{i,i-\alpha}_0\overset{\eqref{e:i k 0}}{=}\sum_{j=i-\alpha}^i Q_{i,j}u_{j-(i-\alpha)}+P_{i,j}y_{j-(i-\alpha)}$;  from \eqref{e:x alpha i i - x zero i i - alpha} it follows that 
\begin{align*}
x^{i,i}_{\alpha}&{=}\sum_{j=i-\alpha}^{i}(Q_{i,j}u_{\alpha-i+j}+P_{i,j}y_{\alpha-i+j})\\&\quad-\sum_{j=i-\alpha+1}^i
(Q_{i,j-1}u_{\alpha-i+j-1}+P_{i,j-1}y_{\alpha-i+j-1})\\
&=Q_{i,i}u_{\alpha}+P_{i,i}y_{\alpha} .
\end{align*}
Together with \eqref{e:i i 0}, this proves that 
\beq\label{e:i i 0-i-1}
\dx{0}{i-1}^{i,i}=Q_{i,i}\du{0}{i-1}+P_{i,i}\dy{0}{i-1}.
\eeq
Now let $\alpha\in[i,T-1]$, and note that
$
x^{i,i}_{\alpha}-x^{i,1}_{\alpha-i+1}=\sum_{k=2}^i(x^{i,k}_{\alpha-i+k}-x^{i,k-1}_{\alpha-i+k-1})
\overset{\eqref{e:i k 1-T}}{=}-\sum_{k=2}^i
(Q_{i,k-1}u_{\alpha-i+k-1}+P_{i,k-1}y_{\alpha-i+k-1}).
$
Use \eqref{e:i 1 1-T} to obtain 
$
x^{i,i}_{\alpha}{=}-\sum_{k=1}^i
(Q_{i,k-1}u_{\alpha-i+k-1}+P_{i,k-1}y_{\alpha-i+k-1}) ,
$
and use \eqref{e:def q and p implies} to conclude that $x^{i,i}_{\alpha}{=}Q_{i,i}u_{\alpha}+P_{i,i}y_{\alpha}$. Conclude that
$
\dx{i}{T-1}^{i,i}=Q_{i,i}\du{i}{T-1}+P_{i,i}\dy{i}{T-1}.
$
Together with \eqref{e:i i 0-i-1}, this proves that \eqref{e:i i 0-T} holds.

Since $\Pi$ is nonsingular, $P_{i,i}$ has full row rank. Therefore, \eqref{e:i i 0-T}  implies that
$
\rs \dy{0}{T-1}\subseteq\rs\bbm \dx{0}{T-1}\\\du{0}{T-1}\ebm.$
It follows from \eqref{e:i 1 1-T} and \eqref{e:i k 1-T} that 
$
\rs \dxfpp \subseteq\rs\bbm \dx{0}{T-1}\\\du{0}{T-1}\ebm .
$
From \eqref{e:state output rsp} we conclude that $\dxf$ is a state for $(\duf,\dyf)$. Note that the number of rows of $\dx{0}{T}$ equals $\sum_{i=0}^{q} is_i$. Since $s_k=\delta_{k-1}-\delta_k$ due to Lemma~\ref{l:si subspaces exist} and since $\delta_q=0$, $\sum_{i=0}^{q} is_i=\sum_{i=0}^{q} \delta_i$. We conclude that $\sy{\sum_{i=0}^{q} \delta_i}\nonempty$. 

To prove the second claim, note that 
$
\syl{q}{\sum_{i=0}^{q} \delta_i}\subseteq\sy{\sum_{i=0}^{q} \delta_i}
$
by definition: it is enough to show that the reverse inclusion holds. To do so, let $\sa\in\sy{\sum_{i=0}^{q} \delta_i}$ and $\ell=\ell(C,A)$. Suppose that $\ell<q$. Since $T-1\geq q$ by  \eqref{e:def q}, we have $T>\ell+1$. Then, Theorem~\ref{t: n lower bound}.\ref{t: n lower bound.1} implies that $\ell\geq q$. This contradicts $\ell<q$. As such, we conclude that $\ell\geq q$. Then, Theorem~\ref{t: n lower bound}.\ref{t: n lower bound.2} implies that $q\geq \ell$, and $\ell=q$. This means that $\sy{\sum_{i=0}^{q} \delta_i}\subseteq\syl{q}{\sum_{i=0}^{q} \delta_i}$. This completes the proof.\EP

\begin{example}[State construction]\label{ex:state construct example}
We  compute a state for the data  in Section~\ref{sec:running example} and  the values $T=5$ and $T=14$ following the procedure in the proof of Theorem~\ref{t:state construct}.

For $T=5$, $\delta_{-1}=\delta_0=2$, and $\delta_k=0$ for $k\in[1,4]$ (see Table~\ref{tab:integers} in Section~\ref{sec:running example}). Use Lemma~\ref{l:si subspaces exist} to conclude that $\dim\calS_0=0$ and $\dim\calS_1=2$. Therefore, $P_{0,j}$ and $Q_{0,j}$ for $j\in\pset{0,1}$ are void matrices and one can choose the following basis matrix for $\calS_1$
\[
\left[\!\begin{array}{c|c|c|c}
\!Q_{1,0}\! &\! Q_{1,1}\! &\! P_{1,0} \!&\! P_{1,1} \! 
\end{array}\!\right]
=
\left[\!\begin{array}{rr|rr|rr|rr}
\!-3\! & \!-1\! &\! -2\! &\! 0 \!& \!1\! & \!0\! &\! 1\! &\! 0\!\\
 \!0 \!& \!-1\! & \! 0\!&\! 0\! &\! 0 \!&\! 0 \!&\! 0\! & \!1\!
\end{array}\!\right].
\]
In view of \eqref{e:i k 0}-\eqref{e:i 1 1-T}, these choices yield the state 
\[
x_{[0,5]}=
\left[\!\begin{array}{rrrrrr}
\!0\!& \!1\! &\! 0\!&\! 1\!&\!0\!&\!1\!\\
\!1\!& \!0\! &\! 0\!&\! 0\!&\!1\!&\! 1\!
\end{array}\!\right].
\]
Solving  \eqref{e:explains}, we obtain the following  explaining system:
\beq\label{e:first exp}
\left[\begin{array}{c|c} A_1 & B_1 \\\hline  C_1 & D_1\end{array}\right]=
\left[\!\begin{array}{rr|rr}
\!-1\!& \!0\! &\! 1\!&\! 1\!\\
\! 0\!&\! 0\! &\! 0\!&\! 1\!\\
\hline
\! 1\!& \!0\! &\! 2\!&\! 0\!\\
\! 0\!& \!1\! &\! 0\!&\! 0\!
\end{array}\!\right].
\eeq
Evidently $ (u_{[0,4]},y_{[0,4]})$ are not informative for system identification: they are explained by a minimal system with 2 states.

For $T=14$, $\delta_{-1}=\delta_0=2$, $\delta_1=1$, and $\delta_k=0$ for $k\in[2,13]$ (see Table~\ref{tab:integers} in Section~\ref{sec:running example}). Moreover, $(u_{[0,13]},y_{[0,13]})$ are informative for system identification for all  values of $L_+$ and $N_+$  in Table~\ref{tab:data info}. As such, we can use Theorem~\ref{t:state construct} to identify an isomorphic system to the true one. We first apply  Lemma~\ref{l:si subspaces exist} and conclude that $\dim\calS_0=0$, $\dim\calS_1=1$, and $\dim\calS_2=1$. Therefore, $P_{0,j}$ and $Q_{0,j}$ for $j\in\pset{0,1}$ are void matrices and one can choose the following bases matrices for $\calS_1$ and $\calS_2$
\[
\left[\!\begin{array}{c|c|c|c}
\!Q_{1,0}\! &\! Q_{1,1}\! &\! P_{1,0} \!&\! P_{1,1} \! 
\end{array}\!\right]
=
\left[\!\begin{array}{rr|rr|rr|rr}
\!-1\!&\!0 \!&\!   -1\!&\!0\!&\!0 \!&\! -1\!&\!1\!&\!0\!
\end{array}\!\right].
\]
\begin{gather*}
\left[\!\begin{array}{c|c|c|c|c|c}
\!Q_{2,0}\! &\! Q_{2,1}\! &\! Q_{2,2}\! &\! P_{2,0} \!&\! P_{2,1} \! &\! P_{2,2}
\end{array}\!\right]\\
\shortparallel\\
\left[\!\begin{array}{rr|rr|rr|rr|rr|rr}
\!0\!&\!-1\!&\! 0\!&\!-1\!&\! 0\!&\! 0\!&\! 0\!&\! 0\!&\! 0\!&\! 0\!&\! 0\!&\! 1\!
\end{array}\!\right].
\end{gather*}
In view of \eqref{e:i k 0}-\eqref{e:i k 1-T}, these choices yield the state 
\[
x_{[0,14]}=
\left[\!\begin{array}{rrrrrrrrrrrrrrr}
\! 1\!&\! 2\!&\! 1\!&\! 1\!&\! 0\!&\! 1\!&\! 2\!&\! 2\!&\! 3\!&\! 2\!&\! 2\!&\! 2\!&\! 2\!&\! 3\!&\! 2\!\\
\! 0\!&\! 0\!&\! 0\!&\! 0\!&\! 1\!&\! 1\!&\! 1\!&\! 1\!&\! 1\!&\! 1\!&\! 1\!&\! 1\!&\! 1\!&\! 0\!&\! 0\!\\
\! 1\!&\! 0\!&\! 0\!&\! 0\!&\! 1\!&\! 2\!&\! 2\!&\! 2\!&\! 2\!&\! 2\!&\! 2\!&\! 2\!&\! 2\!&\! 1\!&\! 0\!
\end{array}\!\right]\; .
\]
Solving \eqref{e:explains} we obtain the  explaining system
\[
\left[\begin{array}{c|c} A_2 & B_2 \\\hline  C_2 & D_2\end{array}\right]=
\left[\!\begin{array}{rrr|rr}
\! 0\!&\! 0\!&\! 1\!&\! 1\!&\! 0\!\\
\! 0\!&\! 0\!&\! 0\!&\! 0\!&\! 1\!\\
\! 0\!&\! 1\!&\! 0\!&\! 0\!&\! 1\!\\\hline
\! 1\!&\! 0\!&\! 0\!&\! 1\!&\! 0\!\\
\! 0\!&\! 0\!&\! 1\!&\! 0\!&\! 0\!
\end{array}\!\right]\; ,
\]
that is isomorphic to the true system.
\end{example}
\section{State-input data Hankel matrices}\label{sec:SI Hankel}
To prove the sufficiency part of Theorem~\ref{t:main} we need auxiliary results on the ranks of state-input data Hankel matrices.

\subsection{On the ranks of state-input data Hankel matrices}

Let $\sa\in\sy{n}$ and $\dxf$ be a state for $\sa$. Define
$$
J_k(x):=\bbm \dx{0}{T-k-1}\\H_{k}(\duf)\ebm
$$
for $k\in[0,T-1]$. Given the structure of $J_k$ we conclude that
\beq\label{e: from jk-1 to jk}
\lk J_{k-1}(x)\times{\bz}_m\subseteq \lk J_k(x)
\eeq
for every $k\in[1,T-1]$. 

We now study the relation of $\rank~H_k$ and $\rank~J_k(x)$. Recall from \eqref{e:hankel data from system and state} that 
$
H_k=\Phi_k J_k(x)$ and $G_k{=}\Psi_k J_k(x).
$
\begin{lemma}\label{l: rank xu H and G} Assume that $T\geq\ell+1$. 
Let $\sa\in\syl{\ell}{n}$ and $\dxf$ be a state for $\sa$, and define $d:=\max(\ell-1,0)$. The following statements hold:  
\benabc
\item\label{l: obs implies jk = hk} If $(C,A)$ is observable, then $\rank J_k(x)=\rank H_k$ for all $k\in[d,T-1]$.
\item\label{l: rank xu H and G.2} If $J_i(x)$ has full row rank for some $i\in[0,T-1]$, then for each $k\in[0,i]$ $J_k(x)$ has full row rank, $\rank H_k=(k+1)m+\rank \Omega_k$, and $\rank G_k=(k+1)m+\rank \Omega_{k-1}$.
\end{enumerate}
\end{lemma}

\begin{proof}
If $(C,A)$ is observable, then $\Phi_k$ has full column rank for $k\geq d$; statement  \ref{l: obs implies jk = hk}  is proved. To prove \ref{l: rank xu H and G.2}, note first that $J_k(x)$ has full row rank whenever $k\in[0,i]$ due to \eqref{e: from jk-1 to jk}. For the rest, observe that \eqref{e:hankel data from system and state} implies $\rank H_k=\rank \Phi_k$ and $\rank G_k=\rank \Psi_k$ whenever $J_k(x)$ has full row rank. From the definitions, we have $\rank \Phi_k=(k+1)m+\rank\Omega_k$ and $\rank \Psi_k=(k+1)m+\rank\Omega_{k-1}$; the claim is proved.\end{proof}

An interesting and useful consequence of Lemma \ref{l: rank xu H and G} is related to the isomorphism property.

\begin{lemma}\label{l:iso prop}
Suppose that $T\geq \ell+1$. Let $\sa\in\syl{\ell}{n}\cap\calO$ and $\dxf$ be a state for $\sa$. If $J_{\ell}(x)$ has full row rank, then $\syl{\ell}{n}\cap\calO$ has the isomorphism property.
\end{lemma}
\begin{proof}
 Lemma~\ref{l: rank xu H and G}.\ref{l: rank xu H and G.2} implies that
$
\rank H_\ell=(\ell+1)m+n.
$
Let $i\in[1,2]$,
$\sai\in \syl{\ell}{n}\cap\calO$ and let $x^i_{[0,T]}$ be a state for $\sai$. Denote $J_k(x^i)$ by $J^i_k$. Because of observability, Lemma~\ref{l: rank xu H and G}.\ref{l: obs implies jk = hk} implies that $J^i_{\ell}$ has full row rank and statement \ref{l: rank xu H and G.2} implies that $J^i_{k}$ has full row rank for all $k\in[0,\ell]$. In particular, $J^i_0$ has full row rank: there exist  $S\in\R^{n\times n}$, $R\in\R^{n\times m}$, $Q\in\R^{m\times n}$, and $P\in\R^{m\times m}$ such that
\begin{equation}\label{eq:relation}
J_{0}^1=\begin{bmatrix}
S&R\\
Q&P
\end{bmatrix}
J_{0}^2. 
\end{equation}
The last $m$ rows of the matrices $J_{0}^i$ are identical; conclude that $\begin{bmatrix}
Q&P-I
\end{bmatrix}J_{0}^2 =0$. Now $J_{0}^2$ has full row rank;  conclude that $Q=0$ and $P=I$. Using the same argument, \eqref{eq:relation} leads to 
$$
J_{\ell}^1=\bbm
S& R & 0 \\
0 & I & 0 \\
0 & 0 & I_{\ell m} 
\ebm
J_{\ell}^2.
$$
Since $J_{\ell}^i$ has full row rank, it follows from \eqref{e:hankel data from system and state} that
$$
\Phi^2_{\ell}=\Phi^1_{\ell}\bbm
S& R & 0 \\
0 & I & 0 \\
0 & 0 & I_{\ell m} 
\ebm\; ,
$$
where $\Phi^i_{\ell}:=\Phi_{\ell}(A_i,B_i,C_i,D_i)$. Using \eqref{e:Phi and Psi}, we see that
\begin{align}
\Omega^2_{\ell}&=\Omega^1_{\ell}S\label{e:omegas rel}\\
\Theta^2_{\ell}&=\Omega^1_{\ell}\bbm R & 0\ebm+\Theta^1_{\ell}\; .\label{e:thetas rel}
\end{align}
Equation \eqref{e:omegas rel} implies that $S$ is nonsingular due to observability whereas \eqref{e:thetas rel} implies that the last $\ell m$ columns of $\Theta^1_{\ell}$ and $\Theta^2_{\ell}$ are identical.
Given their Toeplitz structure, we see that \beq\label{e:d1=d2}
D_1=D_2
\eeq
and $C_1A_1^kB_1=C_2A_2^kB_2$ for  $k\in[0,\ell-2]$. Consequently the entries corresponding to the first $\ell p$ rows and first $m$ columns of $\Theta^1_{\ell}-\Theta^2_{\ell}$ are all zero. Hence, \eqref{e:thetas rel} implies $\Omega^1_{\ell-1}R=0$. Now $R=0$ since $(C_1,A_1)$ is observable. Therefore, \eqref{e:thetas rel} implies $\Theta^2_{\ell}=\Theta^1_{\ell}$. Comparing the first $p$ rows in \eqref{e:omegas rel}, we see that
\beq\label{e:c1=c2}
C_2=C_1 S.
\eeq
From $\Theta^2_{\ell}=\Theta^1_{\ell}$, we have
\beq\label{e:c1a1b1=c2a2b2}
C_2A_2^kB_2=C_1A_1^kB_1
\eeq
for every $k\in[0,\ell-1]$. Note that
$\Omega^1_{\ell-1}S B_2\overset{\eqref{e:omegas rel}}{=}\Omega^2_{\ell-1}B_2\overset{\eqref{e:c1a1b1=c2a2b2}}{=}\Omega^1_{\ell-1}B_1$. From the observability of $(C_1,A_1)$ conclude that
\beq\label{e:b1=b2}
SB_2=B_1.
\eeq
Note that
$
\Omega_{\ell-1}^1(SA_2-A_1S)\overset{\eqref{e:omegas rel}}{=}\Omega_{\ell-1}^2A_2-\Omega_{\ell-1}^1A_1S\overset{\eqref{e:omegas rel}}{=}0.
$
As $(C_1,A_1)$ is observable, we see that $SA_2-A_1S=0$; with \eqref{e:d1=d2}, \eqref{e:c1=c2}, and \eqref{e:b1=b2}, this proves that $\sai$ are isomorphic. Hence $\syl{\ell}{n}\cap\calO$ has the isomorphism property.
\end{proof}

Given \eqref{e: from jk-1 to jk} and the rank-nullity theorem, 
$
\rank J_{k-1}+m\geq \rank J_{k}
$
for all $k\in[1,T-1]$. Such relation between the ranks of two consecutive $J_k$-matrices is related to the controllability of the corresponding explaining system. 

\blem\label{l:uncontrollability}
Let $\sa\in\syl{\ell}{n}$ and $\dxf$ be a state for $\sa$. If for some $k\in[1,T-1]$ $\rank J_{k-1}(x)+m=\rank J_{k}(x)$ and $J_k(x)$ does not have full row rank, then $(A,B)$ is not controllable. 
\elem

\begin{proof}
Let $k\in[1,T-1]$ be such that $\rank J_{k-1}(x)+m=\rank J_{k}(x)$ and $J_k(x)$ does not have full row rank. From the rank-nullity theorem, we have
$
\dim \lk J_{k-1}(x)=\dim \lk J_{k}(x)>0.
$
Then, we see from \eqref{e: from jk-1 to jk} that 
\beq\label{e:lk jk-1=lk jk}
\zset\neq \lk J_{k}(x)=\lk J_{k-1}(x)\times{\bz}_m
\eeq
Now, define $\calA_k\in\R^{(n+(k+1)m)\times (n+(k+1)m)}$ by $
\calA_k:=\begin{bsmallmatrix}
A & B & 0 \\
0 & 0 & I_{km}\\
0_{m,n} & 0 & 0
\end{bsmallmatrix}$. 
The relation $\dxfp=A\dxf+B\duf$ implies that $
\calA_k J_k=
\begin{bsmallmatrix}
\dx{1}{T-k}\\
H_{k-1}(\du{1}{T-1})\\
0_{m,T-k}
\end{bsmallmatrix}$ for every $k\in[0,T-1]$. Note that the matrix $\begin{bmatrix}
\dx{1}{T-k}\\
H_{k-1}(\du{1}{T-1})
\end{bmatrix}$ can be obtained from $J_{k-1}(x)$ by deleting its first column. Hence, we see that
$
(\lk J_{k-1}(x)\times {\bz}_m)\calA_k \subseteq \lk J_k(x)
$
for every $k\in[1,T-1]$. Together with \eqref{e:lk jk-1=lk jk}, this implies that $\zset\neq\lk J_{k-1}(x)\times{\bz}_m$ is left-invariant under $\calA_k$, that is $(\lk J_{k-1}(x)\times{\bz}_m)\calA_k\subseteq\lk J_{k-1}(x)\times{\bz}_m$. Therefore, the subspace $\lk J_{k-1}(x)\times{\bz}_m$ must contain a left-eigenvector of $\calA_k$, say $\zeta\in\C^{1\times (n+(k+1)m)}$. Note that the last $m$ entries of $\zeta$ are zero. Together with the structure of $\calA_k$, this implies that
$\zeta=\begin{bmatrix}
\xi&0_{1,(k+1)m}
\end{bmatrix}$
where $\xi$ is nonzero. Hence, we see that $\xi\begin{bmatrix}
A-\lambda I&B
\end{bmatrix}=0$ for some $\lambda\in\C$. It, then, follows from the Hautus test that $(A,B)$ is uncontrollable. 
\end{proof}

\subsection{Proof of Theorem~\ref{t:main}: sufficiency part}\label{sec:t:mainsuff}
In view of Proposition~\ref{p:streamline} and \eqref{e:trimming wrt la}, proving sufficiency of Theorem~\ref{t:main} requires showing that the conditions \eqref{e:nec suff conditions} imply:
\benabcrm
\item\label{i:si def a with la} $\calE_{[L_-,\la],[N_-,N_+]}\cap\calM=\sy{\nt}\cap\calM$, and
\item\label{i:si def b with la} $\sy{\nt}\cap\calM$ has the isomorphism property.
\end{enumerate}
To this end, we need some preparations.

To begin with, it is clear from the definitions of $\lm$ and $\nm$ that $\calE\cap\calS(\ell,n)=\emptyset$ whenever $\ell<\lm$ or $n<\nm$. Therefore, we see from \eqref{e:lb l} and \eqref{e:lb n} that
\beq\label{e:trim up to lm and nm}
\calE_{[L_-,\la],[N_-,N_+]}=\calE_{[\lm,\la],[\nm,N_+]}.
\eeq

Next, we compute the ranks of the data Hankel matrices $H_k$. Let $\sa\in\sy{\nm}$ and let $\dxf\in\R^{\nm\times(T+1)}$ be a state for $\sa$. Note that $\lm\leq\lt\leq L_+$ and $\lm\leq\ld=N_+-\nm+\lm$. As such, we have $\lm\leq\la=\min(L_+,\ld)$. Due to Theorem~\ref{t:nmin}, $(C,A)$ is observable. Since $m\geq 1$, \eqref{e:lb T} implies that $T\geq \la+1$. Therefore, it follows from Lemma~\ref{l: rank xu H and G}.\ref{l: obs implies jk = hk} that 
$
\rank J_k(x)=\rank H_k
$
for every $k\in[d,\la]$ where $d=\max(\lm-1,0)$. In particular, we see from \eqref{e:main rank condition} that $\rank J_{\la}(x)=(\la+1)m+\nm$ and hence $J_{\la}(x)$ has full row rank. It  follows from Lemma~\ref{l: rank xu H and G}.\ref{l: rank xu H and G.2} that
\beq\label{e: rank of hk if hla is full}
\rank J_k(x)=\rank H_k=(k+1)m+\nm
\eeq
for every $k\in[d,\la]$. We claim that 
\beq\label{e:suff claim 1}
\calE_{[\lm,\la],[\nm,N_+]}\cap\calM\subseteq\calE(\nm)\cap\calM.
\eeq

Suppose first that $N_+=\nm$. Then, \eqref{e:suff claim 1} follows from 
$
\calE_{[\lm,\la],[\nm,\nm]}\subseteq\calE(\nm)$. Suppose now that $N_+>\nm$. Let $\ell\in[\lm,\la]$, $n\in[\nm+1,N_+]$, and $\sah\in\syl{\ell}{n}\cap\calO$. Also, let $\dhxf\in\R^{n\times (T+1)}$ be a state for $\sah$. $\ell\geq 1$ since $(C,A)$ is observable and $n\geq 1$. Since $\ell\geq \lm$, we further see that  $\ell-1\geq d$. Then, Lemma~\ref{l: rank xu H and G}.\ref{l: obs implies jk = hk} and \eqref{e: rank of hk if hla is full} imply that
$
\rank J_{\ell-1}(\hatx)+m=\rank J_\ell(\hatx)=(\ell+1)m+\nm.
$
Since $n>\nm$, Lemma~\ref{l:uncontrollability} implies that $(\hatA,\hatB)$ is not controllable. Therefore, we see that $\calE(\ell,n)\cap\calM=\emptyset$ whenever $\ell\in[\lm,\la]$ and $n\in[\nm+1,N_+]$. Hence, \eqref{e:suff claim 1} holds. 

Note that
$
\calE(\lt,\nt)\subseteq
\calE_{[\lm,\la],[\nm,N_+]}.
$
Then, it follows from \eqref{e:suff claim 1} that
\begin{align}
\calE(\lt,\nt)\cap\calM
&\subseteq\calE_{[\lm,\la],[\nm,N_+]}\cap\calM\notag\\
&\subseteq \syl{\lm}{\nm}\cap\calM\label{e:with inclusion}.
\end{align}
Therefore, we see that 
\beq\label{e:lm=lt and nm=nt}
\lt=\lm\qand\nt=\nm,
\eeq
proving \eqref{e:extra-lm} and \eqref{e:extra-nm}. We conclude from \eqref{e:with inclusion} and Theorem~\ref{t:nmin} that
\begin{align}
\calE(\lt&,\nt)\cap\calM
=\calE_{[\lm,\la],[\nm,N_+]}\cap\calM\notag\\
&= \syl{\lm}{\nm}\cap\calM=\sy{\nm}\cap\calM\label{e:all equal}.
\end{align}
Thus, condition \ref{i:si def a with la} follows from \eqref{e:trim up to lm and nm}, \eqref{e:lm=lt and nm=nt}, and \eqref{e:all equal}. 

To show \ref{i:si def b with la}, note first that \eqref{e: rank of hk if hla is full} and Lemma~\ref{l:iso prop} imply that 
$\calE(\lm,\nm)\cap\calO$ has the isomorphism property. Since the true system is minimal, we see that 
\beq\label{e:min and obs the same}
\calE(\lm,\nm)\cap\calO=\calE(\lm,\nm)\cap\calM.
\eeq
Then, it follows from \eqref{e:all equal} that \ref{i:si def b with la} holds. What remains to be proven is \eqref{e:extra-main}. To do so, note first that Theorem~\ref{t:nmin} implies that $\calE(\nm)=\calE(\lm,\nm)\cap\calO$. Then, we see from \eqref{e:all equal} and \eqref{e:min and obs the same} that 
$
\calE(\nm)=\calE_{[\lm,\la],[\nm,N_+]}\cap\calM.
$
Therefore, \eqref{e:extra-main} follows from 
\eqref{e:trim 2}, \eqref{e:trimming wrt la}, and \eqref{e:trim up to lm and nm}.
\EP

\section{From one explaining system to another}\label{sec: from one to another}
To prove  necessity  in Theorem~\ref{t:main}, we present two ways of computing explaining systems. The first one constructs an explaining system with $n$ states from a given one with $n$ states.
\begin{lemma}\label{l:from one to another}
Suppose that $\sa\in\syl{\ell}{n}\cap\calO$ for some $n\geq \ell\geq 1$. Let $\dxf\in\R^{n\times(T+1)}$ be a state for $\sa$. Also, let $d=\min(\ell,T-1)$, $\xi\in\R^{1\times n}$ and $\eta_i\in\R^{1\times m}$ with $i\in[0,d]$ be such that
\beq\label{e:xi etas in lker}
\xi\dx{0}{T-d-1}+\sum_{i=0}^{d}\eta_i\du{i}{T-d-1+i}=0.
\eeq
Let $0\neq \zeta\in\R^n$ be such that
\beq\label{e:spec zeta}
CA^i\zeta=0\,\,\text{ for }\,\,i\in[0,\ell-2].
\eeq
Define
\begin{alignat}{3}
\hatA&=A+\zeta\xi & \qquad\hatB&=B+E_{-1}\label{e:choice hata}\\
\hatC&=C &
\hatD&=D+CE_0\label{e:choice hatd}
\end{alignat}
where $E_0$ and $E_{-1}$ are determined by the recursion
\beq
E_d=0_{n,m}\qand E_{i-1}=\hatA E_{i}+
\zeta\eta_{i}\,\,\text{ for }\,\,i\in[0,d] \label{e:choice ei}.
\eeq
Then, the following statements hold:
\benabc
\item\label{l:from one to another.1} $\sah\in\calE(\ell,n)\cap\calO$.
\item\label{l:from one to another.2} If $\sa$ and $\sah$ are isomorphic, then
\beni
\item\label{l:iso prop-nec-aux.2} $\eta_i=0$ for every $i\in[0,d]$.
\item\label{l:iso prop-nec-aux.1} $\xi A^iB=0$ for every $i\in[0,n-1]$,
\end{enumerate}
\end{enumerate}

\end{lemma}

\begin{proof}
To prove \ref{l:from one to another.1}, we first show that there exists $\dhx{0}{T}\in\R^{n\times(T+1)}$ satisfying
\begin{align}
\dhx{1}{T}&=\hatA\dhx{0}{T-1}+\hatB\duf\label{e:hatx.1}\\
\dyf&=\hatC\dhx{0}{T-1}+\hatD\duf.\label{e:hatx.2}
\end{align}
We claim that $\dhx{0}{T}\in\R^{n\times(T+1)}$ defined by
\begin{align}
\dhx{0}{T-d}&:=\dx{0}{T-d}-\sum_{i=0}^{d-1}E_i\du{i}{T-d+i}\label{e:def z.1}\\
\hatx_{i+1}&:=\hatA\hatx_i+\hatB u_i\quad\text{for}\quad i\in[T-d,T-1] \label{e:def z.2}
\end{align}
satisfies \eqref{e:hatx.1} and \eqref{e:hatx.2}.

To prove this claim, let $k\in[0,T-d-1]$. Note that
$
\hatx_{k+1}\overset{\eqref{e:def z.1}}{=}x_{k+1}-\sum_{i=0}^{d-1}E_i u_{k+1+i}=Ax_k+Bu_k-\sum_{i=0}^{d-1}E_iu_{k+1+i}$
and 
$\hatA\hatx_k+\hatB u_k
\overset{\eqref{e:def z.1}}{=}\hatA x_k+\hatB u_k-\sum_{i=0}^{d-1}\hatA E_i u_{k+i}$. Using \eqref{e:choice hata} and \eqref{e:choice ei}, one can verify that the difference between these two expressions is
$\zeta\xi x_k+\sum_{i=0}^{d}\zeta\eta_i u_{k+i}$. Therefore, \eqref{e:xi etas in lker} implies that
$
\dhx{1}{T-d}=\hatA\dhx{0}{T-d-1}+\hatB\du{0}{T-d-1}.
$
Together with \eqref{e:def z.2}, this proves \eqref{e:hatx.1}. 

Therefore, it remains to prove \eqref{e:hatx.2}. First, we make a few crucial observations. To begin with, we have 
\beq\label{e:zeta at l-1 nonzero}
CA^{\ell-1}\zeta\neq 0
\eeq
since $(C,A)$ is observable and $\zeta\neq 0$. Also, it follows from \eqref{e:spec zeta} and \eqref{e:choice hata} that
\beq\label{e:from tildeA to A}
C\hatA^i=CA^i\text{ for }i\in[0,\ell-1]
\eeq
and $
C\hatA^{\ell}=CA^{\ell}+CA^{\ell-1}\zeta\xi.
$
Further, observe that
\beq\label{e:c hat a zeta is zero up to l-2}
C\hatA^i\zeta=0
\eeq
for $i\in[0,\ell-2]$ due to \eqref{e:spec zeta} and \eqref{e:from tildeA to A}. Finally, it follows from the recursion \eqref{e:choice ei} that
\beq\label{e:ei as sum}
E_{i}=\sum_{k=0}^{d-i-1}\hatA^k\zeta\eta_{i+k+1}
\eeq
for $i\in[-1,d]$ and from \eqref{e:c hat a zeta is zero up to l-2} that
\beq\label{e:cei}
CE_i=0\text{ for }i\in[1,d].
\eeq

To show \eqref{e:hatx.2}, we first deal with the case $d=0$. Since $d=\min(\ell,T-1)$ and $\ell\geq 1$, we see that $T=1$ in this case. Then, it follows from \eqref{e:def z.1} that $\dhx{0}{1}=\dx{0}{1}$ and from \eqref{e:choice hatd}-\eqref{e:choice ei} that $\hatD=D$. Since $\hatC=C$ due to \eqref{e:choice hatd}, we see that \eqref{e:hatx.2} is readily satisfied if $d=0$. 

Suppose now that $d\geq 1$. Note that
\begin{align}
\hatC&\dhx{0}{T-d}+\hatD\du{0}{T-d}\overset{\eqref{e:choice hatd}}{=}C\dhx{0}{T-d}+\hatD\du{0}{T-d}\notag\\
&\overset{\eqref{e:def z.1}\&\eqref{e:cei}}{=}
C\dx{0}{T-d}-CE_0\du{0}{T-d}+\hatD\du{0}{T-d}\notag\\
&\overset{\eqref{e:choice hatd}}{=}
C\dx{0}{T-d}+D\du{0}{T-d}=\dy{0}{T-d}\label{e:out tot T-k}
\end{align}
where the last equality follows from the fact that $x$ is a state for the explaining system $\sa$. Hence, we see that \eqref{e:hatx.2} is satisfied if $d=1$.

Suppose that $d\geq 2$. In view of \eqref{e:out tot T-k}, what remains to be proven is that
\beq
\dy{T-d+1}{T-1}=\hatC\dhx{T-d+1}{T-1}+\hatD\du{T-d+1}{T-1}.\label{e:after T-k}
\eeq
To do so, let $i\in[1,d-1]$. Define
$
\haty_{T-d+i}:=\hatC \hatx_{T-d+i} +\hatD u_{T-d+i}
$
and
$
\Delta_{T-d+i}:=\haty_{T-d+i}-y_{T-d+i}.
$
Note that
\begin{align*}
\hatx_{T-d+i}&=\hatA^i \hatx_{T-d}+\sum_{j=0}^{i-1}\hatA^{i-j-1}\hatB u_{T-d+j}\\
x_{T-d+i}&=A^i x_{T-d}+\sum_{j=0}^{i-1}A^{i-j-1}Bu_{T-d+j}\\
\hatx_{T-d}&=x_{T-d}-\sum_{j=0}^{d-1}E_ju_{T-d+j}
\end{align*}
where the first equality follows from \eqref{e:def z.2}, the second from the fact that $x$ is a state for the data, and the third from \eqref{e:def z.1}. By using \eqref{e:choice hata}, \eqref{e:choice ei}, \eqref{e:from tildeA to A}, \eqref{e:c hat a zeta is zero up to l-2}, and the fact that $d\leq\ell$, we see that
$
\Delta_{T-d+i}=\sum_{j=0}^{i}C\hatA^{i-j}E_0u_{T-d+j}-\sum_{j=0}^{d-1}C\hatA^iE_ju_{T-d+j}.
$
By using \eqref{e:c hat a zeta is zero up to l-2} and \eqref{e:ei as sum}, one can prove by induction that
\beq
C\hatA^iE_j=C\hatA^{i-j}E_0\label{e:rec based ind.1}
\eeq
for all $j\in[0,d-1]$ and $i\in[j,d-1]$ as well as that
\beq
C\hatA^iE_j=0\label{e:rec based ind.2}
\eeq
for all $j\in[1,d-1]$ and $i\in[0,j-1]$. It follows from \eqref{e:rec based ind.1} that
$
\sum_{j=0}^{i}C\hatA^iE_ju_{T-d+j}=
\sum_{j=0}^{i}C\hatA^{i-j}E_0u_{T-d+j}.
$
Hence, we have
$
\Delta_{T-d+i}=-\sum_{j=i+1}^{d-1}C\hatA^iE_ju_{T-d+j}\overset{\eqref{e:rec based ind.2}}=0.
$
This proves \eqref{e:after T-k} and hence \eqref{e:def z.2} in view of \eqref{e:out tot T-k}. Therefore, we proved that $\sah\in\calE(n)$. Further, it follows from \eqref{e:from tildeA to A} and observability of $(C,A)$ that $(\hatC,\hatA)$ is also observable and $\ell(\hatC,\hatA)=\ell$. Then, we have $\sah\in\calE(\ell,n)\cap\calO$ which proves \ref{l:from one to another.1}.

To prove \ref{l:from one to another.2}, note that
\beq
CE_0=0\qand
CA^iB=\hatC\hatA^i\hatB\quad\text{for all}\quad i\geq 0\label{e:ec0}
\eeq
since the two systems are isomorphic. The latter, together with \eqref{e:from tildeA to A}, implies that
$
CA^iE_{-1}=0
$
for all $i\in[0,\ell-1]$. As $(C,A)$ is observable and $\ell=\ell(C,A)$, we see that $E_{-1}=0$. Since $E_{-1}=\hatA E_0+\zeta\eta_0$, \eqref{e:zeta at l-1 nonzero} and \eqref{e:from tildeA to A} imply that
\beq\label{e:caie0}
0=CA^iE_{-1}=CA^i(\hatA E_0+\zeta\eta_0)=CA^{i+1}E_0
\eeq
for all $i\in[0,\ell-2]$. As $(C,A)$ is observable and $\ell=\ell(C,A)$, \eqref{e:ec0} and \eqref{e:caie0} imply that $E_0=0$. Therefore, we have $\zeta\eta_0=E_{-1}-\hatA E_0=0$. Since $\zeta\neq 0$, this yields $\eta_0=0$. Note that
$
E_0=\sum_{k=0}^{\ell-1}\hatA^k\zeta\eta_{k+1}=0
$
due to \eqref{e:ei as sum}. From \eqref{e:zeta at l-1 nonzero} and \eqref{e:from tildeA to A}, we have $C\hatA^{\ell-1}\zeta\neq 0$. As such, the vectors $\hatA^i\zeta$ with $i\in[0,\ell-1]$ are linearly independent. Then, it follows from $\sum_{k=0}^{\ell-1}\hatA^k\zeta\eta_{k+1}=0$ that $\eta_i=0$ for every $i\in[1,\ell-1]$. Thus, we have proven \ref{l:iso prop-nec-aux.2}.

To prove \ref{l:iso prop-nec-aux.1}, note first that $\hatB=B$ as $E_{-1}=0$. Then, we have 
$
0=C(sI-\hatA)\inv B-C(sI-A)\inv B
=C(sI-\hatA)\inv \zeta\xi (sI-A)\inv B
$
where the first equality follows from isomorphism, the second is evident. Since $C(sI-\hatA)\inv \zeta$ is a nonzero column vector, we see that $\xi (sI-A)\inv B=0$. This proves \ref{l:iso prop-nec-aux.1}.
\end{proof}

\begin{example}
We illustrate Lemma~\ref{l:from one to another} with the data in Section~\ref{sec:running example}. Consider the explaining system \eqref{e:first exp} for  $(u_{[0,4]},y_{[0,4]})$. Note that \eqref{e:xi etas in lker} is satisfied with $\xi=\begin{bmatrix}
-1&-1
\end{bmatrix}$, $ \eta_0=\begin{bmatrix}
0 & 1
\end{bmatrix} $, and $ \eta_1=\begin{bmatrix}
1 & 0
\end{bmatrix} $. Since $\ell=1$ for  \eqref{e:first exp}, $\zeta=\col(1,1)$ satisfies \eqref{e:spec zeta}. Applying Lemma~\ref{l:from one to another}, we obtain the explaining system
\[
\left[\begin{array}{c|c} \hatA_1 & \hatB_1 \\\hline  \hatC_1 & \hatD_1\end{array}\right]=
\left[\!\begin{array}{rr|rr}
\!-2\!& \!-1\! &\! -2\!&\! 2\!\\
\!-1\!&\! -1\! &\! -2\!&\! 2\!\\
\hline
\! 1\!& \!0\! &\! 3\!&\! 0\!\\
\! 0\!& \!1\! &\! 1\!&\! 0\!
\end{array}\!\right]\; ,
\]
with state $x_{[0,5]}=
\left[\!\begin{array}{rrrrrr}
\!-1\!& \!0\! &\! -1\!&\! 1\!&\!0\!&\!1\!\\
\!0\!& \!-1\! &\! -1\!&\! 0\!&\!1\!&\! 1\!
\end{array}\!\right]$. \hfill \qed\end{example}
Using Lemma~\ref{l:from one to another}, we  state a necessary condition for the isomorphism property. 
\begin{lemma}\label{l:iso prop-nec}
Suppose that $n\geq\ell\geq 0$ and $\sa\in\syl{\ell}{n}\cap\calM$. Let $\dxf\in\R^{n\times(T+1)}$ be a state for $\sa$. If 
$\syl{\ell}{n}\cap\calM$ has the isomorphism property, then $T\geq\ell+(\ell+1)m+n$ and $J_{\ell}(x)$ has full row rank.
\end{lemma}
\begin{proof}
Suppose first that $\ell=n=0$. Since $\syl{0}{0}\cap\calM$ has the isomorphism property, $\duf=J_{0}(x)$ must have full row rank and hence $T\geq m$. 

Now, suppose that $n\geq\ell\geq 1$. 
Let $d=\min(\ell,T-1)$, $\xi\in\R^{1\times n}$ and $\eta_i\in\R^{1\times m}$ with $i\in[0,d]$ be vectors such that $
\begin{bmatrix}
\xi&\eta_0&\cdots&\eta_d
\end{bmatrix}\in\lk J_d(x)$. Also, let $\zeta_0$ be a nonzero vector be such that $CA^i\zeta_0=0$ for $i\in[0,\ell-2]$. For $\varepsilon>0$, let $\sahe$ denote the explaining system obtained from Lemma~\ref{l:from one to another} by taking $\zeta=\varepsilon\zeta_0$. Since $(A,B)$ is controllable, so is $(\hatA_\varepsilon,\hatB_\varepsilon)$ for all sufficiently small $\varepsilon$. Hence, we see that $\sahe\in\calE(\ell,n)\cap\calM$ for some $\epsilon>0$. Since $\calE(\ell,n)\cap\calM$ has the isomorphism property and $(A,B)$ is controllable, Lemma~\ref{l:from one to another}.\ref{l:from one to another.2} implies that $\xi=0$ and $\eta_i=0$ for all $i\in[0,d]$. This means that $J_{d}(x)$ has full row rank. Since $T\geq1$, $m\geq 1$, and $n\geq 1$, $J_{T-1}(x)$ has at least 2 rows and exactly 1 column. As such, it cannot have full row rank. Then, we see that $d=\min(\ell,T-1)\neq T-1$. Therefore, $T-1>\ell$, $d=\ell$, and $J_{\ell}(x)$ has full row rank. This implies that $T\geq\ell+(\ell+1)m+n$.\end{proof}

We now show how to construct an explaining system from a given one by increasing the state dimension and  the lag.
\blem\label{l:lag extension}
Suppose that $n\geq 1$, $A\in\Rnn$, and $C\in\Rpn$ are such that $(C,A)$ is observable. Denote $\ell=\ell(C,A)$. Let $\zeta\in\R^n$ be such that 
\beq\label{e:C A zeta relation}
CA^i\zeta=0\quad\forall\,i\in[0,\ell-2]\quad\text{and}\quad CA^{\ell-1}\zeta\neq 0.
\eeq
Also, let $n'\geq 1$, $A'\in\R^{n'\times n'}$ and $C'\in\R^{1\times n'}$ be such that $(C',A')$ is observable. Then, the pair
$$
(\barC,\barA):=\left(\bbm
C & 0_{n,n'}
\ebm,
\bbm
A &\zeta C'\\
0 &A'
\ebm
\right)
$$
is observable and $\bar\ell:=\ell(\barC,\barA)=\ell+n'$. Moreover, if $\zeta'\in\R^{n'}$ satisfies
\beq\label{e:prime zeta}
C'(A')^i\zeta'=0\,\,\forall\,i\in[0,n'-2]\,\,\text{and}\,\, C'(A')^{n'-1}\zeta'\neq 0,
\eeq
then
\beq\label{e:bar zeta}
\barC\barA^i\bbm0\\\zeta'\ebm=0\,\,\forall\,i\in[0,\bar\ell-2]\,\,\text{and}\,\, \barC\barA^{\bar\ell-1}\bbm0\\\zeta'\ebm\neq 0.
\eeq
\elem

\begin{proof}
By direct inspection, we see that
$\barC\barA^k=\bbm CA^k & \Xi_k\ebm$
where $\Xi_0=0$, $\Xi_{k+1}=CA^k\zeta C'+\Xi_kA'$ for all $k\geq 0$. By using \eqref{e:C A zeta relation}, we further see that
\ben[label=(\roman*),ref=(\roman*)]
\item\label{i:up to l-1} $\Xi_{k}=0$ for all $k\in[0,\ell-1]$, and
\item\label{i:up at l} $\Xi_{\ell}=CA^{\ell-1}\zeta C'$.
\een
Let $\bar\Omega_{k}$, $\Omega_{k}$, and $\Omega'_{k}$ be the $k$-th observability matrices of the pairs $(\barC,\barA)$, $(C,A)$, and $(C',A')$, respectively. We claim that
\beq\label{e:ranks of obs matrices}
\rank \bar\Omega_{\ell+i}=n+\rank\Omega'_{i}
\eeq
for all $i\geq 0$. To show this, let $i\geq 0$. Note that $\bar\Omega_{\ell+i}$ is of the form
\beq\label{e:spec form}
\bar\Omega_{\ell+i}=\begin{bsmallmatrix} \Omega_{\ell-1}& 0\\
* & \Xi_{\ell}\\
* & \Xi_{\ell+1}\\
\vdots & \vdots\\
* & \Xi_{\ell+i}
\end{bsmallmatrix}.
\eeq
From \ref{i:up at l} and \eqref{e:C A zeta relation},
it follows that
\beq\label{e:from Xi to Omega prime}
\ker\begin{bsmallmatrix} 
\Xi_{\ell}\\
\Xi_{\ell+1}\\
 \vdots\\
\Xi_{\ell+i}
\end{bsmallmatrix}=\ker\Omega'_{i} \text{ and }\rank\begin{bsmallmatrix} 
\Xi_{\ell}\\
\Xi_{\ell+1}\\
 \vdots\\
\Xi_{\ell+i}
\end{bsmallmatrix}=\rank\Omega'_{i}.
\eeq
Since $(C,A)$ is observable and $\ell(C,A)=\ell$, $\rank\Omega_{\ell-1}=n$. Therefore, we see from \eqref{e:spec form} that \eqref{e:ranks of obs matrices} holds. Since $C'\in\R^{1\times n'}$ and $(C',A')$ is observable, we have $\ell(C',A')=n'$. Then, \eqref{e:ranks of obs matrices} implies that $\rank\bar\Omega_{\ell+n'-1}=n+n'$ and hence that $(\barC,\barA)$ is observable. It also follows from \eqref{e:ranks of obs matrices} that $\rank\bar\Omega_{\ell+n'-2}<n+n'$. This means that  $\ell(\barC,\barA)=\ell+n'$. Further, if $\zeta'$ satisfies \eqref{e:prime zeta} then $\zeta'\in\ker\Omega'_{n'-2}$ and $\zeta'\not\in\ker\Omega'_{n'-1}$. From \eqref{e:spec form} and \eqref{e:from Xi to Omega prime}, it follows that $\bbm0\\\zeta'\ebm\in\ker\bar\Omega_{\ell+n'-2}$ and $\bbm0\\\zeta'\ebm\not\in\ker\bar\Omega_{\ell+n'-1}$. Hence, \eqref{e:bar zeta} holds.\end{proof}

\begin{lemma}\label{l:extending towards la}
Suppose that $n\geq\ell\geq 0$ and $\sa\in\syl{\ell}{n}\cap\calM$. Let $\dxf\in\R^{n\times(T+1)}$ be a state for $\sa$. If 
$\mu\geq 1$ and $\syl{\ell+\mu}{n+\mu}\cap\calM=\emptyset$, then $T\geq\ell+\mu+(\ell+\mu+1)m+n$ and $J_{\ell+\mu}(x)$ has full row rank.
\end{lemma}

\begin{proof} Let $\lambda\in\R$, $A'_\lambda\in\R^{\mu\times \mu}$ be the Jordan block with the eigenvalue $\lambda$, $C':=e_1^T$, and $\zeta':=e_{\mu}$ where $e_i$ denotes the $i$th standard basis vector of $\R^\mu$. Clearly, $(C',A'_\lambda)$ is observable and $\ell(C',A'_\lambda)=\mu$. In addition, we have that
$
C'(A'_\lambda)^i\zeta'=0\quad\forall\,i\in[0,\mu-2]$ and $ C'(A'_\lambda)^{\mu-1}\zeta'=1.
$

We claim that $J_d(x)$ has full row rank where $d=\min(\ell+\mu,T-1)$. To prove it  we  consider the cases $n=0$ and $n\geq 1$.

For the case $n=0$, we have that $\ell=0$, $x$ is a void matrix, and $\dyf=D\duf$ for some $D\in\R^{p\times m}$. Therefore, for every nonzero $\theta\in\R^p$, $\dzf:=0_{\mu\times(T+1)}$ is a state for $\sahnz\in\calE(\mu,\mu)\cap\calO$. Let $\eta_i$ with $i\in[0,d]$ be such that $\begin{bmatrix}
\eta_0&\cdots&\eta_d
\end{bmatrix}\in\lk J_d(x)$. Clearly, we have $\begin{bmatrix}
0_{1,\mu}&\eta_0&\cdots&\eta_d
\end{bmatrix}\in\lk J_d(z)$. Define
\begin{alignat*}{3}
\hatA_{\lambda}&=A_{\lambda}' & \qquad\hatB_{\lambda}&=E_{-1}\\
\hatC&=\theta C' &
\hatD_{\lambda}&=D+\theta C'E_0
\end{alignat*}
where 
$
E_d=0\,\,\,\text{and}\,\,\, E_{i-1}=\hatA_{\lambda} E_{i}+
\zeta'\eta_{i}\,\,\,\text{for}\,\,\,i\in[0,d] .
$
Since $d=\min(\mu,T-1)$ for this case, it follows from Lemma~\ref{l:from one to another}.\ref{l:from one to another.1} that $\sahnzz\in\calE(\mu,\mu)\cap\calO$. Since $\calE(\mu,\mu)\cap\calM=\emptyset$ due to the hypothesis, $(\hatA_{\lambda},\hatB_{\lambda})$ is uncontrollable. From the fact that $\hatA_{\lambda}=A'_{\lambda}$ is a Jordan block, we see that $(\zeta')^TA'_{\lambda}=\lambda (\zeta')^T$ and $(\zeta')^TE_{-1}=0$. Since $E_{-1}=\sum_{k=0}^{d}(A'_{\lambda})^k\zeta'\eta_{k}$ due to \eqref{e:ei as sum}, we see that $\sum_{k=0}^{d}\lambda^k\eta_k=0$. As $\lambda$ is an arbitrary real number, we conclude that $\eta_i=0$ for every $i\in[0,d]$ and hence $J_d(x)$ has full row rank.

For the case $n\geq 1$, let $\zeta\in\R^n$ be as in \eqref{e:C A zeta relation} and define
$
\barC:=\bbm
C & 0_{n,\mu}
\ebm$, $\barA_{\eps,\lambda}:=
\bbm
A &\eps\zeta C'\\
0 &A'_\lambda
\ebm$, $\barB:=\begin{bmatrix}
B\\0_{\mu,m}
\end{bmatrix}
$
for $\eps>0$. Then, it follows from Lemma~\ref{l:lag extension} that $(\barC,\barA_{\eps,\lambda})$ is observable, $\ell(\barC,\barA_{\eps,\lambda})=\ell+\mu=:\bar\ell$, 
\beq\label{e:eps lambda lag extension}
\barC\barA_{\eps,\lambda}^i\bbm0\\\zeta'\ebm=0\,\,\forall\,i\in[0,\bar\ell-2]\,\,\text{and}\,\, \barC\barA_{\eps,\lambda}^{\bar\ell-1}\bbm0\\\zeta'\ebm\neq 0.
\eeq
Note that $\sab\in\syl{\ell+\mu}{n +\mu}\cap\calO$ and $\dzf:=\begin{bsmallmatrix}
\dxf\\0_{\mu\times(T+2)}\end{bsmallmatrix}$ is a state for $\sab$. Let $d=\min(\ell+\mu,T-1)$. Also, let $\xi\in\R^{1\times n}$, $\eta_i$ with $i\in[0,d]$ be such that $\begin{bmatrix}
\xi&\eta_0&\cdots&\eta_d
\end{bmatrix}\in\lk J_d(x)$. Clearly, we have $\begin{bmatrix}
\xi&0_{1,\mu}&\eta_0&\cdots&\eta_d
\end{bmatrix}\in\lk J_d(z)$. Define
\begin{alignat*}{3}
\hatA_{\eps,\lambda}&=\barA_{\eps,\lambda}+\bbm 0\\\zeta'\ebm\bbm\xi&0\ebm & \qquad\hatB_{\eps,\lambda}&=\barB+E_{-1}\\
\hatC&=\barC &
\hatD_{\eps,\lambda}&=D+\barC E_0
\end{alignat*}
where 
$
E_d=0\,\,\,\text{and}\,\,\, E_{i-1}=\hatA_{\eps,\lambda} E_{i}+
\bbm 0\\\zeta'\ebm\eta_{i}\,\,\,\text{for}\,\,\,i\in[0,d] .
$
Then, it follows from Lemma~\ref{l:from one to another}.\ref{l:from one to another.1} that $\sahb\in\syl{\ell+\mu}{n +\mu}\cap\calO$. From the hypothesis, we know that $(\hatA_{\eps,\lambda},\hatB_{\eps,\lambda})$ is uncontrollable. By taking the limit as $\eps$ tends to zero, we conclude that $(\hatA_{0,\lambda},\hatB_{0,\lambda})$ is uncontrollable as well. Note that
\[
\hatA_{0,\lambda}=\begin{bmatrix}
A & 0\\
\zeta'\xi & A'_{\lambda}
\end{bmatrix}\qand 
\hatB_{0,\lambda}=
\begin{bmatrix}
B\\0_{\mu,m}
\end{bmatrix}+F_{-1}
\]
where $F_d=0$ and $F_{i-1}=\hatA_{0,\lambda}F_i+\bbm 0\\\zeta'\ebm\eta_i$ for $i\in[0,d]$. Suppose that $\lambda$ is not an eigenvalue of $A$. Then, every left eigenvector of $\hatA_{0,\lambda}$ corresponding to an eigenvalue of $A$ must be of the form 
$
\begin{bmatrix}
v&0
\end{bmatrix}
$
where $v\in\C^{1\times n}$. From 
$
(\zeta')^TA_\lambda=\lambda(\zeta')^T
$, 
we see that left eigenvectors of
$
\hatA_{0,\lambda}
$
corresponding to the eigenvalue $\lambda$ are nonzero multiples of 
$
\begin{bmatrix}
\xi(\lambda I-A)\inv&(\zeta')^T
\end{bmatrix}
$. Since $(A,B)$ is controllable but
$
(\hatA_{0,\lambda},\hatB_{0,\lambda})
$
is uncontrollable, it follows from the Hautus test that
$
\begin{bmatrix}
\xi(\lambda I-A)\inv&(\zeta')^T
\end{bmatrix}
\hatB_{0,\lambda}
=0
$.
Since 
$
F_{-1}=\sum_{k=0}^d\hatA_{0,\lambda}^k\bbm 0\\\zeta'\ebm\eta_k
$,
we see that
$
\xi(\lambda I-A)\inv B+\sum_{k=0}^d \lambda^k\eta_k=0.
$
Since this equality  holds for all $\lambda\in\R$ that are not  eigenvalues of $A$, we  conclude that $\eta_i=0$ for $i\in[0,d]$ and 
$
\xi(\lambda I-A)\inv B=0
$. Since $(A,B)$ is controllable we conclude that $\xi=0$ and that $J_d(x)$ has full row rank. 

To prove that $J_{\ell+\mu}(x)$ has full row rank, note that $J_{T-1}(x)$ has at least $ 2 $ rows and exactly $ 1 $ column since $T\geq1$, $m\geq 1$, and $n\geq 1$. As such, it cannot have full row rank. Then, we see that $d=\min(\ell+\mu,T-1)\neq T-1$. Therefore, $T-1>\ell+\mu$, $d=\ell+\mu$, and $J_{\ell+\mu}(x)$ has full row rank. The latter implies that $T\geq\ell+\mu+(\ell+\mu+1)m+n$.\end{proof}

\subsection{Proof of Theorem~\ref{t:main}: necessity part}\label{sec:t:mainnec}
Assume that the data are informative for system identification in $\calS_{[L_-,L_+],[N_-,N_+]}\cap\calM$. Let 
$
\sa\in\syl{\lt}{\nt}\cap\calM
$
and let
$
x\in\R^{n\times(T+1)}
$
be a state for 
$
\sa
$.
Since 
$
\sy{\nt}\cap\calM
$
has the isomorphism property, we have
$
\sy{\nt}\cap\calM=\syl{\lt}{\nt}\cap\calM
$.
Then, Lemma~\ref{l:iso prop-nec} implies that 
\beq\label{e: T lower with lt}
T\geq\lt+(\lt+1)m+\nt
\eeq
and 
$J_{\lt}(x)$ has full row rank whereas Lemma~\ref{l: rank xu H and G}.\ref{l: rank xu H and G.2} implies that
$\delta_k=\rho_k
$
for every
$
k\in[0,\lt]
$. 
As $\rho_{\lt-1}\geq 1$ due to Lemma~\ref{p: sum rk}.\ref{p: sum rk.bb}, we see that $q\geq \lt$. Then, Theorem~\ref{t:nmin} implies that $\lm\geq\lt$. Since the reverse inequality readily follows from the definition of $\lm$ in \eqref{e:def lm}, we have $\lt=\lm$. Further, Theorem~\ref{t:nmin} and Lemma~\ref{p: sum rk}.\ref{p: sum rk.ee} imply that $\nt=\nm$. Then, \eqref{e:lb l} and \eqref{e:lb n} follow from \eqref{e:bounds on lt and nt}.

If $\la=\lt$, \eqref{e:lb T} readily follows from \eqref{e: T lower with lt} and \eqref{e:main rank condition} follows from $\rank H_{\lt}=\rank J_{\lt}(x)$ and $J_{\lt}(x)$ having full row rank. Suppose that $\la>\lt$. Note that the informativity of the data for system identification within $\calS_{[L_-,L_+],[N_-,N_+]}\cap\calM$ implies that $\syl{\lt+\mu}{\nt+\mu}\cap\calM=\emptyset$ where $\mu=\la-\lt$. Then, Lemma~\ref{l:extending towards la} implies that \eqref{e:lb T} holds and $J_{\la}(x)$ has full row rank. Since $\rank H_{\la}=\rank J_{\la}(x)$, we see that \eqref{e:main rank condition} holds.\EP
\subsection{Proposition~\ref{prop:wetal} vs. Theorem~\ref{t:main}}\label{sec:prop to main thm}
We want to show that the condition 
\beq\label{e:wetal cond}
\rank H_{L_++N_+-1}\left(\begin{bmatrix}
\duf\\\dyf
\end{bmatrix}\right)=(L_++N_+)m+\nt
\eeq
appearing in Proposition~\ref{prop:wetal} implies the conditions \eqref{e:lb l}, \eqref{e:lb n}, and \eqref{e:main rank condition} of Theorem~\ref{t:main}. Let $\sa\in\sy{\nm}$ and let $\dxf\in\R^{\nm\times(T+1)}$ be a state for $\sa$. Since $L_++N_+-1\geq \lm$, we see from \eqref{e:hankel data from system and state} and \eqref{e:wetal cond} that $\rank J_{L_++N_+-1}(x)=(L_++N_+)m+\nt$. This implies that $\nt\leq\nm $. Since the reverse inequality holds due to \eqref{e:def nm}, we see that $\nm=\nt$. Then, it follows from Theorem~\ref{t:nmin} that $\lm=\lt$. Therefore, \eqref{e:lb l}, \eqref{e:lb n}, and \eqref{e:main rank condition} are satisfied.

\section{Conclusions}\label{sec:conc}
We stated necessary and sufficient conditions for informativity for system identification in the class of minimal ISO systems whose lag and state dimension lie  between given lower/upper bounds (see Theorem~\ref{t:main}). To establish such result we obtained some intermediate ones of independent interest, most prominently the iterative construction of a state sequence and a corresponding ISO model in the proof of Theorem~ \ref{t:state construct}. 

We aim to apply the concept of informativity for system identification in a model class (see Definition \ref{def:IWPKC}) to other classes than minimal systems, e.g. to dissipative systems. We also plan to work on the application of the informativity concept to identification in the behavioral framework, where interesting  results (see \cite{Markovsky23}) have recently appeared.

\bibliographystyle{IEEEtran}
\bibliography{references}

\end{document}

%% file: ka-newcommands.tex
\DeclareMathOperator{\col}{col}

\DeclareMathOperator{\rank}{rank}

\let\leq\leqslant
\let\geq\geqslant
\let\emptyset\varnothing

\newcommand{\calA}{\ensuremath{\mathcal{A}}}

\newcommand{\calE}{\ensuremath{\mathcal{E}}}

\newcommand{\calM}{\ensuremath{\mathcal{M}}}

\newcommand{\calO}{\ensuremath{\mathcal{O}}}

\newcommand{\calS}{\ensuremath{\mathcal{S}}}

\newcommand{\calV}{\ensuremath{\mathcal{V}}}

\newcommand{\hatx}{\ensuremath{\hat{x}}}
\newcommand{\haty}{\ensuremath{\hat{y}}}

\newcommand{\hatA}{\ensuremath{\hat{A}}}
\newcommand{\hatB}{\ensuremath{\hat{B}}}
\newcommand{\hatC}{\ensuremath{\hat{C}}}
\newcommand{\hatD}{\ensuremath{\hat{D}}}

\newcommand{\barA}{\ensuremath{\bar{A}}}
\newcommand{\barB}{\ensuremath{\bar{B}}}
\newcommand{\barC}{\ensuremath{\bar{C}}}

\newcommand{\bmat}{\begin{matrix}}
\newcommand{\emat}{\end{matrix}}
\newcommand{\bbm}{\begin{bmatrix}}
\newcommand{\ebm}{\end{bmatrix}}
\newcommand{\bbma}{\begin{bmatrix*}[r]}
\newcommand{\ebma}{\end{bmatrix*}}
\newcommand{\bpm}{\begin{pmatrix}}
\newcommand{\epm}{\end{pmatrix}}
\newcommand{\bvm}{\begin{vmatrix}}
\newcommand{\evm}{\end{vmatrix}}
\newcommand{\bse}{\begin{subequations}}
\newcommand{\ese}{\end{subequations}}
\newcommand{\beq}{\begin{equation}}
\newcommand{\eeq}{\end{equation}}
\newcommand{\ben}{\renewcommand{\labelenumi}{\arabic{enumi}.}
\renewcommand{\theenumi}{\arabic{enumi}}\begin{enumerate}}

\newcommand{\een}{\end{enumerate}}

\newcommand{\bena}{\renewcommand{\labelenumi}{\alph{enumi}.}
\renewcommand{\theenumi}{\alph{enumi}}\begin{enumerate}}

\newcommand{\eena}{\end{enumerate}}

\newcommand{\bit}{\begin{itemize}}
\newcommand{\eit}{\end{itemize}}
\newcommand{\bthe}{\begin{theorem}}
\newcommand{\ethe}{\end{theorem}}
\newcommand{\blem}{\begin{lemma}}
\newcommand{\elem}{\end{lemma}}
\newcommand{\bprop}{\begin{proposition}}
\newcommand{\eprop}{\end{proposition}}
\newcommand{\bex}{\begin{example}}
\newcommand{\eex}{\end{example}}
\newcommand{\bas}{\begin{assumption}}
\newcommand{\eas}{\end{assumption}}
\newcommand{\bre}{\begin{remark}}
\newcommand{\ere}{\end{remark}}
\newcommand{\bcor}{\begin{corollary}}
\newcommand{\ecor}{\end{corollary}}
\newcommand{\bdfn}{\begin{definition}}
\newcommand{\edfn}{\end{definition}}
\newcommand{\bcon}{\begin{conjecture}}
\newcommand{\econ}{\end{conjecture}}

\newcommand{\eps}{\ensuremath{\varepsilon}}

\newcommand{\inv}{\ensuremath{^{-1}}}
\newcommand{\pset}[1]{\ensuremath{\{#1\}}}

\newcommand{\zset}{\ensuremath{\pset{0}}}

\newcommand{\set}[2]{\ensuremath{\left\{#1: #2\right\}}}

\newcommand{\R}{\ensuremath{\mathbb R}}
\newcommand{\C}{\ensuremath{\mathbb C}}
\newcommand{\N}{\ensuremath{\mathbb N}}
\newcommand{\Z}{\ensuremath{\mathbb Z}}

\newcommand{\EP}{\hspace*{\fill} $\square$}
\newcommand{\qand}{\quad\text{ and }\quad}
\newcommand{\nonempty}{\neq\emptyset}

%% file: ka-newcommands-extra.tex
\newcommand{\Rnn}{\R^{n\times n}}

\newcommand{\Rnm}{\R^{n\times m}}
\newcommand{\Rpn}{\R^{p\times n}}
\newcommand{\Rpm}{\R^{p\times m}}
\newcommand{\Rpp}{\R^{p\times p}}

\newcommand{\bz}{\bm{0}}

\newcommand{\ta}{A_{\mathrm{true}}}
\newcommand{\tb}{B_{\mathrm{true}}}
\newcommand{\tc}{C_{\mathrm{true}}}
\newcommand{\td}{D_{\mathrm{true}}}

\newcommand{\du}[2]{u_{[#1,#2]}}
\newcommand{\dy}[2]{y_{[#1,#2]}}
\newcommand{\dx}[2]{x_{[#1,#2]}}

\newcommand{\dhx}[2]{\hatx_{[#1,#2]}}
\newcommand{\dhxf}[2]{\hatx_{[0,T]}}

\newcommand{\duf}{u_{[0,T-1]}}
\newcommand{\dyf}{y_{[0,T-1]}}
\newcommand{\dxf}{x_{[0,T]}}
\newcommand{\dxfp}{x_{[1,T]}}
\newcommand{\dxfpp}{x_{[1,T]}}

\newcommand{\dxfm}{x_{[0,T-1]}}

\newcommand{\dzf}{z_{[0,T]}}

\newcommand{\nt}{n_{\mathrm{true}}}
\newcommand{\lt}{\ell_{\mathrm{true}}}
\newcommand{\nm}{n_{\mathrm{min}}}
\newcommand{\lm}{\ell_{\mathrm{min}}}
\newcommand{\la}{L^{\mathrm{a}}_+}
\newcommand{\ld}{L^{\mathrm{d}}_+}

\newcommand{\hk}[1]{H_{#1}}
\newcommand{\hhk}[1]{G_{#1}}

\newcommand{\spk}{\calS_{\mathrm{pk}}}
\newcommand{\sy}[1]{\calE(#1)}
\newcommand{\syl}[2]{\calE(#1,#2)}

\newcommand{\sln}{\syl{\ell}{n}}

\newcommand{\sa}{\left[\begin{smallmatrix} A & B \\ C & D\end{smallmatrix}\right]}
\newcommand{\strue}{\left[\begin{smallmatrix} \ta & \tb \\ \tc & \td\end{smallmatrix}\right]}

\newcommand{\sah}{\left[\begin{smallmatrix} \hatA & \hatB \\ \hatC & \hatD\end{smallmatrix}\right]}
\newcommand{\sahe}{\left[\begin{smallmatrix} \hatA_\varepsilon & \hatB_\varepsilon \\ \hatC_\varepsilon & \hatD_\varepsilon\end{smallmatrix}\right]}
\newcommand{\sai}{\left[\begin{smallmatrix} A_i & B_i \\ C_i & D_i\end{smallmatrix}\right]}

\newcommand{\sab}{\left[\begin{smallmatrix} \barA_{\varepsilon,\lambda} & \barB \\ \barC & D\end{smallmatrix}\right]}
\newcommand{\sahb}{\left[\begin{smallmatrix} \hatA_{\varepsilon,\lambda} & \hatB_{\varepsilon} \\ \barC & \hatD_{\varepsilon}\end{smallmatrix}\right]}
\newcommand{\sahnz}{\left[\begin{smallmatrix} A_{\lambda}' & 0_{\mu,m} \\ \theta C' & D\end{smallmatrix}\right]}
\newcommand{\sahnzz}{\left[\begin{smallmatrix} \hatA_{\lambda} & \hatB_{\lambda} \\ \hatC & \hatD\end{smallmatrix}\right]}
\newcommand{\struebig}{\left[\begin{array}{c|c} \ta & \tb \\\hline \tc & \td\end{array}\right]}

\newcommand{\lr}[1]{\left(#1\right)}
\DeclareMathOperator{\lk}{lker}
\DeclareMathOperator{\rs}{rsp}

%% file: 2024-05-29-arxiv.bbl
\begin{thebibliography}{10}
\providecommand{\url}[1]{#1}
\csname url@samestyle\endcsname
\providecommand{\newblock}{\relax}
\providecommand{\bibinfo}[2]{#2}
\providecommand{\BIBentrySTDinterwordspacing}{\spaceskip=0pt\relax}
\providecommand{\BIBentryALTinterwordstretchfactor}{4}
\providecommand{\BIBentryALTinterwordspacing}{\spaceskip=\fontdimen2\font plus
\BIBentryALTinterwordstretchfactor\fontdimen3\font minus
  \fontdimen4\font\relax}
\providecommand{\BIBforeignlanguage}[2]{{%
\expandafter\ifx\csname l@#1\endcsname\relax
\typeout{** WARNING: IEEEtran.bst: No hyphenation pattern has been}%
\typeout{** loaded for the language `#1'. Using the pattern for}%
\typeout{** the default language instead.}%
\else
\language=\csname l@#1\endcsname
\fi
#2}}
\providecommand{\BIBdecl}{\relax}
\BIBdecl

\bibitem{Willems86a}
J.~C. Willems, ``From time series to linear system - {P}art {I}. {F}inite
  dimen{\-}sional linear time invariant systems,'' \emph{Automatica}, vol.~22,
  no.~5, pp. 561--580, 1986.

\bibitem{Willems86b}
------, ``From time series to linear system - {P}art {II}. {E}xact modelling,''
  \emph{Automatica}, vol.~22, no.~6, pp. 675--694, 1986.

\bibitem{Willems87}
------, ``From time series to linear system - {P}art {III}. {A}pproximate
  modelling,'' \emph{Automatica}, vol.~23, no.~1, pp. 87--115, 1987.

\bibitem{Moonen1989}
M.~Moonen, B.~De~Moor, L.~Vandenberghe, and J.~Vandewalle, ``On- and off-line
  identification of linear state-space models,'' \emph{International Journal of
  Control}, vol.~49, no.~1, pp. 219--232, 1989.

\bibitem{Willems05}
J.~C. Willems, P.~Rapisarda, I.~Markovsky, and B.~{De Moor}, ``A note on
  persistency of excitation,'' \emph{Systems \& Control Letters}, vol.~54,
  no.~4, pp. 325--329, 2005.

\bibitem{Markovsky2007}
I.~Markovsky and P.~Rapisarda, ``On the linear quadratic data-driven control,''
  in \emph{2007 European Control Conference (ECC)}, 2007, pp. 5313--5318.

\bibitem{Markovsky2008}
------, ``Data-driven simulation and control,'' \emph{International Journal of
  Control}, vol.~81, no.~12, pp. 1946--1959, 2008.

\bibitem{Rapisarda2011}
P.~Rapisarda and H.~L. Trentelman, ``Identification and data-driven model
  reduction of state-space representations of lossless and dissipative systems
  from noise-free data,'' \emph{Automatica}, vol.~47, no.~8, pp. 1721--1728,
  2011.

\bibitem{Yang2013}
H.~Yang and S.~Li, ``A new method of direct data-driven predictive controller
  design,'' in \emph{9th Asian Control Conference, ASCC 2013}, 2013.

\bibitem{Maupong2017}
T.~M. Maupong and P.~Rapisarda, ``Data-driven control: A behavioral approach,''
  \emph{Systems \& Control Letters}, vol. 101, pp. 37--43, 2017, {J}an C.
  Willems Memorial Issue, Volume 2.

\bibitem{Coulson19}
J.~Coulson, J.~Lygeros, and F.~Dorfler, ``Data-enabled predictive control: in
  the shallows of the {D}ee{PC},'' in \emph{18th European Control Conference
  (ECC)}, 2019, pp. 307--312.

\bibitem{DePerzik20}
C.~{De Persis} and P.~Tesi, ``Formulas for data-driven control: stabilization,
  optimality, and robustness,'' \emph{IEEE Transactions on Automatic Control},
  vol.~65, no.~3, pp. 909--924, 2020.

\bibitem{Berberich21}
J.~Berberich, J.~Kohler, M.~A. Muller, and F.~Allgower, ``Data-driven model
  predictive control with stability and robustness guarantees,'' \emph{IEEE
  Transactions on Automatic Control}, vol.~66, no.~4, pp. 1702--1717, 2021.

\bibitem{Waarde20b}
H.~J. {van Waarde}, C.~{De Persis}, M.~K. Camlibel, and P.~Tesi, ``Willems'
  fundamental lemma for state-space systems and its extension to multiple
  datasets,'' \emph{IEEE Control Systems Letters}, vol.~4, no.~3, pp. 602--607,
  2020.

\bibitem{Markovsky23a}
I.~Markovsky, E.~Prieto-Araujo, and F.~Dorfler, ``On the persistency of
  excitation,'' \emph{Automatica}, vol. 147, p. 110657, 2023.

\bibitem{Yu21}
Y.~Yu, S.~Talebi, H.~J. van Waarde, U.~Topcu, M.~Mesbahi, and B.~Acikmese, ``On
  controllability and persistency of excitation in data-driven control:
  extensions of {W}illems' fundamental lemma,'' in \emph{60th IEEE Conference
  on Decision and Control (CDC)}, 2021, pp. 6485--6490.

\bibitem{Mishra21}
V.~K. Mishra, I.~Markovsky, and B.~Grossmann, ``Data-driven tests for
  controllability,'' \emph{IEEE Control Systems Letters}, vol.~5, no.~2, pp.
  517--522, 2021.

\bibitem{Lopez22}
V.~G. Lopez and M.~A. Muller, ``On a continuous-time version of {W}illems'
  lemma,'' in \emph{IEEE 61st Conference on Decision and Control (CDC)}, 2022,
  pp. 2759--2764.

\bibitem{Rapisarda23a}
P.~Rapisarda, M.~K. Camlibel, and H.~J. {van Waarde}, ``A `fundamental lemma'
  for continuous-time systems, with applications to data-driven simulation,''
  \emph{Systems \& Control Letters}, vol. 179, p. 105603, 2023.

\bibitem{Rapisarda23b}
P.~Rapisarda, M.~K. Camlibel, and H.~J. van Waarde, ``A persistency of
  excitation condition for continuous-time systems,'' \emph{IEEE Control
  Systems Letters}, vol.~7, pp. 589--594, 2023.

\bibitem{Lopez24}
V.~Lopez, M.~M\"uller, and P.~Rapisarda, ``An input-output continuous-time
  version of {W}illems' lemma,'' \emph{IEEE Control Systems Letters}, to
  appear.

\bibitem{Berberich23}
J.~Berberich, A.~Iannelli, A.~Padoan, J.~Coulson, F.~Dorfler, and F.~Allgower,
  ``A quantitative and constructive proof of {W}illems' fundamental lemma and
  its implications,'' in \emph{2023 American Control Conference (ACC)}, 2023,
  pp. 4155--4160.

\bibitem{Coulson23}
J.~Coulson, H.~J. {van Waarde}, J.~Lygeros, and F.~Dorfler, ``A quantitative
  notion of persistency of excitation and the robust fundamental lemma,''
  \emph{IEEE Control Systems Letters}, vol.~7, pp. 1243--1248, 2023.

\bibitem{Ferizbegovic21}
M.~Ferizbegovic, H.~Hjalmarsson, P.~Mattsson, and T.~B. Schon, ``Willems'
  fundamental lemma based on second-order moments,'' in \emph{60th IEEE
  Conference on Decision and Control (CDC)}, 2021, pp. 396--401.

\bibitem{Meijer23}
T.~J. Meijer, S.~A.~N. Nouwens, V.~S. Dolk, and W.~P. M.~H. Heemels, ``A
  frequency-domain version of {W}illems' fundamental lemma,''
  \emph{https://arxiv.org/abs/2311.15284v1}, 2023.

\bibitem{Waarde21}
H.~J. van Waarde, ``Beyond persistent excitation: Online experiment design for
  data-driven modeling and control,'' \emph{IEEE Control Systems Letters},
  vol.~6, pp. 319--324, 2022.

\bibitem{Schmitz22}
P.~Schmitz, T.~Faulwasser, and K.~Worthmann, ``{W}illems' fundamental lemma for
  linear descriptor systems and its use for data-driven output-feedback
  {MPC},'' \emph{IEEE Control Systems Letters}, vol.~6, pp. 2443--2448, 2022.

\bibitem{Alsalti21}
M.~Alsalti, J.~Berberich, V.~G. Lopez, F.~Allgower, and M.~A. Muller,
  ``Data-based system analysis and control of flat nonlinear systems,'' in
  \emph{60th IEEE Conference on Decision and Control (CDC)}, 2021, pp.
  1484--1489.

\bibitem{Verhoek21}
C.~Verhoek, R.~Toth, S.~Haesaert, and A.~Koch, ``Fundamental lemma for
  data-driven analysis of linear parameter-varying systems,'' in \emph{60th
  IEEE Conference on Decision and Control (CDC)}, 2021, pp. 5040--5046.

\bibitem{Faulwasser22}
T.~Faulwasser, R.~Ou, G.~Pan, P.~Schmitz, and K.~Worthmann, ``Behavioral theory
  for stochastic systems? {A} data-driven journey from {W}illems to {W}iener
  and back again,'' \emph{Annual Reviews in Control}, vol.~55, pp. 92--117,
  2023.

\bibitem{Markovsky05}
I.~Markovsky, J.~C. Willems, and B.~De~Moor, ``State representations from
  finite time series,'' in \emph{Proceedings of the 44th IEEE Conference on
  Decision and Control}, 2005, pp. 832--835.

\bibitem{Waarde20a}
H.~J. van Waarde, J.~Eising, H.~L. Trentelman, and M.~K. Camlibel, ``Data
  informativity: a new perspective on data-driven analysis and control,''
  \emph{IEEE Transactions on Automatic Control}, vol.~65, no.~11, pp.
  4753--4768, 2020.

\bibitem{Waarde23-CSM}
H.~J. van Waarde, J.~Eising, M.~K. Camlibel, and H.~L. Trentelman, ``The
  informativity approach to data-driven analysis and control,'' \emph{IEEE
  Control Systems Magazine}, vol.~43, no.~6, pp. 32--66, 2023.

\bibitem{vanOverschee1996}
P.~van Overschee and B.~de~Moor, \emph{Subspace Identification for Linear
  Systems: Theory, Implementation, Applications}.\hskip 1em plus 0.5em minus
  0.4em\relax Kluwer Academic Publishers, 1996.

\bibitem{Markovsky05c}
I.~Markovsky, J.~Willems, and B.~De~Moor, ``State representations from finite
  time series,'' in \emph{Proceedings of the 44th IEEE Conference on Decision
  and Control}, 2005, pp. 832--835.

\bibitem{Markovsky23}
I.~Markovsky and F.~Dorfler, ``Identifiability in the behavioral setting,''
  \emph{IEEE Transactions on Automatic Control}, vol.~68, no.~3, pp.
  1667--1677, 2023.

\bibitem{Camlibel24}
M.~K. Camlibel, H.~J. {van Waarde}, and P.~Rapisarda, ``The shortest experiment
  for linear system identification,'' in \emph{26th International Symposium on
  Mathematical Theory of Networks and Systems}, 2024, submitted for
  publication.

\end{thebibliography}
